%% file: ParaID_Bayes_arxivv3.tex
\documentclass[a4paper,11pt]{article}

\usepackage{amsmath,amssymb,amsthm,bbm}
\usepackage{bm}
\usepackage{paralist,enumitem} 
\usepackage{hyperref,url}
\usepackage{color}
\usepackage[margin=30mm]{geometry}
\usepackage{placeins}
\usepackage{subfig}
\usepackage{cleveref}
 
\usepackage{pgfplots}  
\pgfplotsset{compat=1.9}
\usepackage{xcolor}
\usepackage{tikz}
\usetikzlibrary{plotmarks,backgrounds}
\usepgfplotslibrary{external}
\input{images/gaussian.pgf}

   \tikzexternaldisable

\usepackage{todonotes}

\newcommand{\change}[1]{#1}

 \newcommand{\MAP}{{\mathrm{MAP}}}

\newcommand{\N}{\mathbb{N}}
\newcommand{\R}{\mathbb{R}}
\newcommand{\Z}{\mathbb{Z}}

\newcommand{\G}{\mathcal{G}}
\newcommand{\Ga}{\mathcal{G}^{a}}
\newcommand{\Gb}{\mathcal{G}^{b}}
\newcommand{\etaa}{\eta^{a}}

\newcommand{\ya}{y^{a}}
\newcommand{\yb}{y^{b}}
\newcommand{\Ya}{Y_{a}}
\newcommand{\Yb}{Y_{b}}
\newcommand{\La}{L_{a}}
\newcommand{\Lb}{L_{b}}
\newcommand{\Ind}{\bm{1}}
\newcommand{\la}{l^{a}}
\newcommand{\lb}{l^{b}}
\newcommand{\Phia}{\Phi_{a}}
\newcommand{\Phib}{\Phi_{b}}
\newcommand{\Gammaa}{\Gamma_{a}}
\newcommand{\Gammab}{\Gamma_{b}}
\newcommand{\sigmaa}{\sigma_{a}}
\newcommand{\sigmab}{\sigma_{b}}

\renewcommand{\div}{ \, \mathrm{div} \, }

\newcommand{\dx}{\, dx}

\newcommand{\pd}{\partial}

\newcommand{\pdnu}{\pd_{n}}

\newcommand{\abs}[1]{\left| #1 \right|}
\newcommand{\norm}[1]{\| #1 \|}
\newcommand{\normY}[1]{\| #1 \|_Y}
\newcommand{\normYa}[1]{\| #1 \|_{\Ya}}
\newcommand{\normYb}[1]{\| #1 \|_{\Yb}}
\newcommand{\inner}[2]{\langle #1 , #2 \rangle}
\newcommand{\eps}{\varepsilon}
\newcommand{\Lap}{\Delta}
\newcommand{\PP}{\mathcal{P}}
\newcommand{\CC}{\mathcal{C}}
\renewcommand{\AA}{\mathbb{A}}
\newcommand{\BB}{\mathbb{B}}
\newcommand{\WW}{\mathbb{W}}

\theoremstyle{plain}
\newtheorem{theorem}{Theorem}[section]

\newtheorem{lemma}[theorem]{Lemma}
\newtheorem{corollary}[theorem]{Corollary}
\newtheorem{remark}[theorem]{Remark}

\numberwithin{equation}{section}

\begin{document}

\title{Bayesian parameter identification in Cahn--Hilliard models for biological growth\,\footnote{The first and fourth author gratefully acknowledge the support by the Deutsche Forschungsgemeinschaft (DFG) through the International Research Training Group IGDK 1754 ``Optimization and Numerical Analysis for Partial Differential Equations with Nonsmooth Structures".  The third and fourth author gratefully acknowledge the support by the DFG and Technische Universit\"at M\"unchen through the International Graduate School of Science and Engineering within project 10.02 BAYES.}}

\author{Christian Kahle \footnotemark[3] 
\and Kei Fong Lam \footnotemark[2] 
\and Jonas Latz \footnotemark[3] 
\and Elisabeth Ullmann \footnotemark[3]}

\date{ }
 
\maketitle

\renewcommand{\thefootnote}{\fnsymbol{footnote}}
\footnotetext[3]{Zentrum Mathematik, Technische Universit\"at M\"unchen, 85748 Garching bei M\"unchen, Germany \tt (\{Christian.Kahle,jlatz,Elisabeth.Ullmann\}@ma.tum.de).}
\footnotetext[2]{Department of Mathematics, The Chinese University of Hong Kong, Shatin, N.T., Hong Kong \tt (kflam@math.cuhk.edu.hk)}

\begin{abstract}
{
We consider the inverse problem of parameter estimation in a diffuse interface model for tumour growth.  
The model consists of a fourth-order Cahn--Hilliard system and contains three phenomenological parameters: the tumour proliferation rate, the nutrient consumption rate, and the chemotactic sensitivity.  
We study the inverse problem within the Bayesian framework and construct the likelihood and noise for two typical observation settings.
One setting involves an infinite-dimensional data space where we observe the full tumour.
In the second setting we observe only the tumour volume, hence the data space is finite-dimensional.
We show the well-posedness of the posterior measure for both settings, building upon and improving the analytical results in [C. Kahle and K.F. Lam, Appl. Math. Optim. (2018)].
A numerical example involving synthetic data is presented in which the posterior measure is numerically approximated by the sequential Monte Carlo approach with tempering.  
}
%
%
%
\end{abstract}



\section{Introduction}\label{sec:intro}
Until recently, the discipline of medical science relied heavily on experiments to obtain statistical data as a basis for understanding the behaviour of complex biomedical systems, and to design new drugs for the treatment of diseases.  
The advent of high-performance computing, big data and bioinformatics, coupled with advances in mathematical and statistical theories and methodologies, has lead to the emergence of patient-specific diagnoses, and treatment driven by computational models.

A complex biomedical phenomenon that is still not fully understood is the growth of cancer.  A tumour is a mass of tissue that arises when certain inhibition proteins in the cells have been switched off by genetic mutations.
This leads to unregulated growth that is limited only by the amount of nutrients in the surrounding environment.  Tumours display characteristics that are fundamentally different from normal cells.  Tumour cells are able to ignore \textit{apoptosis} (programmed cell death) signals, remain elusive to attacks from the immune system, and, most dangerously, have the ability to induce the growth of new blood vessels towards itself (\textit{angiogenesis}).
This leads to the spreading of cancer to other parts of the body, and the formation of secondary tumours (\textit{metastasis}). 

The study of tumour growth can be roughly divided according to the physical and chemical phenomena occuring at three scales \cite{Oden16}: the tissue scale which is commonly observed in experiments involving movement of cells (such as metastasis and growth into the extracellular matrix) and nutrient diffusion; the cellular scale consisting of activities and interactions between individual cells such as mitosis and  the activation of receptors; and sub-cellular scale where genetic mutations and DNA degradation occur.  
We focus on the tissue-scaled phenomena, as they are the first to be detected in a routine diagnosis, and can be described fairly well with help of continuum models consisting of differential equations.

Since the seminal work in \cite{Burton} and \cite{Green} where simple mathematical models for tumour growth are employed, there has been an explosion in the number of models proposed for modelling the multiscale nature of cancer, see for instance \cite{CL,DS,Oden16} and the references cited therein.  
The diversity of model variants reflects the difficulties when we try to identify key biological phenomena that are responsible for experimental observations.

As metastasis is an important hallmark of cancer, we restrict our attention to continuum models that can capture such events.  
Continuum models often rely on a mathematical description to distinguish tumour tissue from healthy host tissues.
To be able to capture metastasis the models have to allow for some form of topological change of the separation layers between the tumour and the host tissues.  
The classical description represents the separation layers as idealised hypersurfaces, known also as the \emph{sharp interface approach}.
In this case complicated boundary conditions have to be imposed to model the mass transfer between tumour and host cells.  
Unfortunately, in the event of metastasis, the separation layers can no longer be represented as a hypersurface, 
and the classical sharp interface approach breaks down.  
To overcome this difficulty some authors proposed a \emph{diffuse interface approach} (see for example \cite{CLLW,CL,Frie,GLNS,GLSS,Hawkins,Lima,Wise} and the references cited therein) which is well-known for being able to handle changes in the topology.

\subsection{Diffuse interface model}
Let $\Omega \subset \R^{2}$ denote a bounded domain with boundary $\pd \Omega$.  
For an arbitrary but fixed constant $T > 0$ we define $Q := \Omega \times (0,T)$ and $\Sigma := \pd \Omega \times (0,T)$.  
Let $n$ denote the outer unit normal vector of $\pd \Omega$, and $\pdnu f := \nabla f \cdot n$ is the co-normal derivative of a function $f$.  
We assume that a tumour is surrounded by healthy tissue and that a nutrient, whose concentration we denoted by $\sigma$, is present which is consumed only by the tumour cells for proliferation.  
We use the difference in volume fractions between the tumour cells and the healthy host cells, denoted by $\varphi \in [-1,1]$, to indicate their location.  
In particular, the set $\{\varphi(x,t) = 1\}$ is the tumour region and the set $\{\varphi(x,t) = -1 \}$ is the healthy tissue.  
The diffuse interface model we consider is a simplification of the general model derived in \cite{GLSS}, however, the model retains some important characteristics.
The model equations read
\begin{subequations}\label{CH}
\begin{alignat}{4}
\label{varphi} \varphi_t & = \div (m(\varphi) \nabla (\mu - \chi \sigma)) + \PP f(\varphi) g(\sigma) && \text{ in } Q, \\
\label{mu} \mu &= - \eps \Lap \varphi + \eps^{-1}\Psi'(\varphi) && \text{ in } Q, \\
\label{sigma} \sigma_t &= \Lap \sigma - \CC h(\varphi) \sigma && \text{ in } Q, \\
\label{noflux} 0 & = \pdnu \varphi = \pdnu \mu = \pdnu \sigma && \text{ on } \Sigma,  \\
\varphi(0)& = \varphi_0, \quad \sigma(0) = \sigma_0 && \text{ in } \Omega,
\end{alignat}
\end{subequations}
where in the above, $\mu$ is the auxiliary variable known as the chemical potential associated with $\varphi$, $m(\varphi)$ is a non-degenerate concentration-dependent mobility, and $\Psi'(\varphi)$ is the derivative of a double-well potential $\Psi(\varphi)$ that has two equal minima at $\varphi = \pm 1$.  
The classical example is $\Psi(s) = (s^2 - 1)^2$. 
The constant $\eps >0$ can be viewed as the thickness of the separation layers $\{\abs{\varphi(x,t)} < 1 \}$ between the tumour and the host tissues, so that in the limit $\eps \to 0$, one can recover sharp interface models for tumour growth, see \cite[\S 3]{GLSS} for more details.

When $\chi = \PP = 0$, \eqref{varphi}--\eqref{mu} reduces to the classical Cahn--Hilliard equation.  
To model tumour proliferation, we include a source term of the form $\PP f(\varphi) g(\sigma)$, where $\PP \geq 0$ can be interpreted as a proliferation rate, $f(\varphi)$ is a smooth indicator function of the growing tumour front, and $g(\sigma)$ captures how the nutrient is used for growth.  
The additional flux term $\nabla (\chi \sigma) = \chi \nabla \sigma$ in the divergence allows the tumour cells to exhibit chemotactic behaviour and move towards regions of high nutrient.  
Hence, $\chi \geq 0$ has the meaning of a chemotactic sensitivity.  
Lastly, we use the reaction-diffusion equation \eqref{sigma} to model the evolution of the nutrient, and the source term $\CC h(\varphi) \sigma$ accounts for how nutrient is consumed by the cells.  
The parameter $\CC \geq 0$ is the consumption rate, and the function $h(\varphi)$ is a smooth interpolation between $h(-1) =0$ and $h(1) = 1$, so that only the tumour cells 
consume the nutrient.
We refer the reader to the earlier work \cite{KL} for some justifications on studying the reduced model \eqref{CH} as opposed to the full model in \cite{GLSS}.

\subsection{Parameter identification}\label{subsec:param}

The main scientific interest of the present work lies in the estimation of certain model parameters given observations on the evolution of the tumour.  This is also known as \textit{model calibration}.  
For meaningful applications of the model in medical research, it is important to calibrate the model parameters in preparation for a comparison between simulations and experimental observations.
This process is known as \textit{model validation}.  
The evaluation and refinement of the mathematical model, such as accounting for fine-scale phenomena, elimination of slow processes, and changes to boundary conditions, can then be made to improve our understanding of tumour growth.

In this paper we focus on the identification of the proliferation rate $\PP$, the chemotactic sensitivity $\chi$, and the consumption rate $\CC$, which we assume to be \emph{non-negative constants}.  It is reasonable to assume constant parameters here, as spatial variations of the mechanisms of proliferation, consumption and chemotactic movement are handled by the functions $m$, $f$, $h$ and $g$.  
It is well known that initially homogeneous tumour cells will eventually develop \change{heterogeneous} growth behaviour, manifesting in the appearance of quiescent cells (those that are not proliferating but are still alive) and necrotic cells (those that are dead).  
Since our model \eqref{CH} accounts for the evolution of a young tumour before the onset of \change{heterogeneous} growth behaviour, it is appropriate to assume that all tumour cells proliferate, move, and consume nutrients at the same rate.  
Since $\eps$ does not influence the growth of the tumour cells, in the theoretical analysis we set $\eps = 1$, and focus on the identification of $\PP, \chi$ and $\CC$.


The classical framework for parameter estimation uses Tikhonov regularization. 
The goal is to obtain the \emph{optimal} parameters such that the mismatch between the model output and the data is minimised, typically in some form of $L^{2}$-distance.  
For the model \eqref{CH} this has been done in the recent work \cite{KL}. 
However, the robustness of the optimal parameters with respect to uncertainty in the measurements is only investigated numerically by optimizing for a range of given noise levels.  To address this issue we employ a modern framework for parameter identification, namely statistical inversion using the Bayesian methodology.
This allows us to incorporate the uncertainties associated with measurements and the relative probabilities of different optimal parameters given by the data.  

For the observations, we consider data obtained from (two-dimensional) snapshots of the tumour, either at a single time instance or at several time instances.  
All observations are inherently polluted by some form of noise. 
We model the noisy observations using a set of linear functionals $\{l_{j}\}_{j=1}^{J}$ of the variable $\varphi$ (which serves as an indicator of the tumour location), where $J \in \N$ is fixed.
Specifically, we assume that $l_{j}: \varphi \mapsto l_{j}(\varphi) \in Y^{(j)}$ for some separable Banach space $Y^{(j)}$.  
The data space is $Y := \prod_{j=1}^J Y^{(j)}$.
The noisy observations, denoted by $\{y_{j}\}_{j=1}^{J} \in Y$, $y_{j} \in Y^{(j)}$, are expressed as
\begin{align}\label{obs:data}
y_{j} = l_{j}(\varphi) + \eta_{j}, \quad j = 1, \dots, J,
\end{align}
with the observational noise denoted by $\{\eta_{j}\}_{j=1}^{J}$, $\eta_{j} \in Y^{(j)}$.  
In this work we analyze two choices for $Y^{(j)}$, $J$ and $l_{j}$:\smallskip

\begin{enumerate}
\item[$(a)$] Let $J = 1$ and $Y^{(1)} := L^2(\Omega)$ with $l_{1}(\varphi) := \varphi(\cdot, T)$; \smallskip
\item[$(b)$] Let $J > 1$ and $Y^{(j)} = \mathbb{R}$ with $l_{j}(\varphi) = \frac{1}{2} \int_{\Omega} (\varphi(x, t_{j})+1 )\mathrm{d}x$, $j=1,\dots,J$, for an increasing sequence $\{t_{j}\}_{j=1}^{J}$ in the time interval $[0,T]$.\smallskip
\end{enumerate}

In setting $(a)$ we observe the tumour everywhere in the spatial domain $\Omega$ at a particular point in time.
In setting $(b)$ we observe the volume of the tumour at $J$ sequential points in time.  
The latter setting is an adaptation of the problem setting in \cite{Collis_tutorial}.
The form of $l_{j}(\varphi)$ in $(b)$ is motivated by the fact that the tumour region is represented by the set $\{\varphi(x,t) = 1\}$.
Hence the function $\frac{1}{2}(\varphi+1)$ takes the value $+1$ inside the tumour, and $0$ outside, which can be viewed as a smoothed indicator function for the tumour.

\subsection{Bayesian inversion and main contributions}
We collect the parameters $(\PP, \chi, \CC)$ in a variable $u \in X$, where $X$ denotes a subspace of a separable Banach space associated with the forward model \eqref{CH}.
Moreover, we define the noise vector $\eta:=(\eta_{1}, \dots, \eta_{J})^\top \in Y$ and the observation vector $y:=(y_{1}, \dots, y_{J})^\top \in Y$.
Then, equations \eqref{obs:data} read
\begin{align}\label{obs:data:com}
y = \G(u) + \eta,
\end{align} 
where we introduced the \emph{forward response operator} $\G : X \to Y$ with
\begin{align*}
\G(u) = (l_{1}(\varphi(u)), \dots, l_{J}(\varphi(u))).
\end{align*}
The operator  $\G$ is the composition of the forward solution operator 
$u = (\PP, \chi, \CC) \mapsto \varphi$ and the observation functionals $\varphi \mapsto l_{j}(\varphi)$.  
Whenever we discuss a particular setting, we denote the forward response operator in setting $(a)$ by $\Ga$ and in setting $(b)$ by $\Gb$.  
Throughout the paper, we use the notation $\{\etaa, \la, \ya\}$ to denote the noise, observation functional and observations for setting $(a)$, and analogously for setting $(b)$.  We use the notation $\G$, $\eta$, $l$ and $y$ if we do not distinguish between setting $(a)$ and $(b)$.
  
We treat the parameter vector $u$ and the observational noise $\eta$ as stochastically independent random variables, i.e., $u$ and $\eta$ are measurable maps from an underlying probability space $(\Omega', \mathcal{F}',\mathbb{P})$ to the spaces $X$ and $Y$, respectively, and satisfy
$\mathbb{P}(u \in A_X, \eta \in A_Y) = \mathbb{P}(\eta \in A_Y)\cdot \mathbb{P}(\eta \in A_Y)$
for measurable subsets $A_X \subseteq X$, $A_Y \subseteq Y$.
Moreover, $u$ is distributed according to a \emph{prior (measure)}, and we assume that the noise is distributed as $\eta \sim \mathrm{N}(0, \Gamma)$.  
In this framework, the observation of the data can be considered as an event $\{\mathcal{G}(u)+\eta = y\}$.
Hence the parameter identification task consists in the computation of the  \emph{posterior (measure)}, that is, the conditional measure of $u$ given that $\{\mathcal{G}(u)+\eta = y\}$ occurs.

The Bayesian framework for parameter identification has been investigated for the Gompertzian tumour \change{spheroid} model in
\cite{Achilleos,Collis_tutorial,Paek,Pat}, and for reaction-diffusion models in \cite{Chang,Le,Meghdadi,Menze}.
The use of phase field type models such as \eqref{CH} for tumour growth modelling and prediction have been suggested first in \cite{Oden2010}, where the beginning steps of the development of Bayesian methods for statistical calibration, 
model validation and uncertainty quantification are also outlined.  
In \cite{HawkinsUQ}, which can be viewed as our closest counterpart in terms of model complexity, 
the authors consider three models of a \change{type} similar to \eqref{CH}, and calibrate the parameters $\PP$, $\chi$, and the diffusion coefficient of a
quasi-static nutrient using synthetic data.  
Subsequently, in \cite{Lima,Lima3,Oden13} the identification of criteria for selecting the most plausible model among several classes of models for given data is discussed.
However, to the best of our knowledge, the well-posedness of the Bayesian inverse problem \eqref{obs:data:com} for parameter identification with a phase field tumour model remains unaddressed, since previous works \cite{HawkinsUQ,Lima,Lima3,Oden2010,Oden13} focused on the implementation aspects of the Bayesian framework.  

Due to the emergence of phase field models as a new modelling tool in biological sciences, and to provide analytical support for the framework of model selection, calibration and validation proposed in \cite{Oden2010}, we study the Bayesian inverse problem \eqref{obs:data:com}. 
The main contribution of our work is to establish a first result on the well-posedness of the Bayesian inverse problem where the underlying model is a nonlinear system of equations with a fourth order Cahn--Hilliard component.  

Our main result, formulated in \Cref{thm:wellposed} below, concerns the well-definedness of the posterior measure and its Lipschitz continuity with respect to the data $y$ in the Hellinger distance.  
The proof relies on the strong well-posedness of solutions to \eqref{CH}, and for our present purposes we require some new extensions of the established results for the forward solution operator $u \mapsto \varphi$.  
Moreover, we recall that setting $(a)$ involves infinite-dimensional function spaces, and the measurement error is modelled by Gaussian white noise, which requires a generalised concept of random fields.  
Hence, the construction of the likelihood function and the proof of posterior well-posedness in this setting may be of independent interest.

We complement the analytical investigations by numerical experiments based on synthetic data for setting $(a)$.
We remark that real-world measurement data is available for setting $(b)$ from \cite{Collis_data} for the volume of 10 individual spherical growing tumours.
However, it turns out that the data contains three-dimensional dynamics (see \cite[\S 2.2]{Collis_tutorial}), and so the data from \cite{Collis_data} is inconsistent with our two-dimensional analysis.  
Although one can perform the parameter identification for setting $(b)$ with \eqref{CH}, preliminary tests lead us to conclude that the computational effort involved in such a nonlinear PDE model for the relatively simple setting $(b)$ is not justified.  
Hence, we focus solely on setting $(a)$ for our numerical computations.

Typically, Bayesian inverse problems are approached with importance sampling \cite{Agapiou2015} or Markov chain Monte Carlo \cite{Beskos2008,Cotter2013}.  
In preliminary tests we observed that both these methods require a prohibitively large number of evaluations of the expensive forward model that is discretized using finite elements on a 2D square domain.  
One model evaluation takes approximately 5 minutes in serial on a workstation, and due to the parabolic nature of the model equations, a single model evaluation is orders of magnitude more expensive compared to one evaluation of the classical elliptic equation often considered in Bayesian inverse problems (see e.g.\ \cite{Dashti2011, Latz2018}) on the same square domain.  
Unfortunately, the required spatial adaptivity prohibits standard domain decomposition approaches to accelerate the computations.

For this reason we approximate the posterior measure by sequential Monte Carlo (SMC) with tempering, see e.g.\ \cite{Beskos2015} for the application of SMC in the context of elliptic equations in three dimensions, and \cite{Kantas2014} for the Navier--Stokes equations, respectively.
We point out that sequential Monte Carlo with tempering has not been employed for the calibration of tumour models; note that Markov chain Monte Carlo (MCMC) methods were employed in previous works \cite{HawkinsUQ,Lima,Lima3,Oden2010,Oden13}.  
We observe that SMC offers attractive features. 
Indeed, SMC allows for parallel model evaluations (as opposed to classical MCMC) and for highly informative data (as opposed to importance sampling).
We select the tempering parameter adaptively to maintain a fixed effective sample size in every SMC update step throughout the algorithm, see \cite{Beskos2016} for a discussion and analysis of this approach.

The remaining part of the paper is organized as follows.  
In \Cref{sec:forward} we construct the likelihood for both problem settings, and establish some properties of the operator $\mathcal{G}$.  \Cref{sec:Bayesian} is dedicated to the well-posedness of the posterior measure.  
In \Cref{sec:FE} we describe a fully discrete finite element approximation of  \eqref{CH}.
In \Cref{sec:SMC} we review SMC with tempering to solve the Bayesian inverse problem.  
A numerical example is presented in \Cref{sec:num}, and we discuss further research topics in \Cref{sec:Discussion}.

\section{Likelihood and properties of the forward model}\label{sec:forward}
\subsection{Model for noise and data likelihood} \label{ssec:likelihoods}
The \emph{(data) likelihood} $L(y|u)$ is a conditional probability density function (PDF) on the data space $Y$.
The observational data is a realisation of the \emph{data-generating measure} with the PDF $L(y|u^\dagger)$ where $u^\dagger$ is fixed.
Given data and likelihood the task of a statistical analysis is the identification of $u^\dagger$.
We discuss the modelling of observational noise and data likelihood in the following paragraphs.
We use the notation $\mathrm{N}(\mu,\Gamma)$ for a \textit{Gaussian measure} with mean function $\mu$ and covariance operator $\Gamma$.

For setting $(a)$, let $\Gammaa := \sigmaa^2\mathrm{Id}_{\Ya}$ be a scaled identity operator on $\Ya$, where $\sigmaa^2 \neq 0$.  We assume that the observational noise satisfies 
$\etaa \sim \mathrm{N}(0,\Gammaa)$.   It is important to note that $\Ya=L^2(\Omega)$ is an infinite-dimensional space, and that the noise $\etaa$ is a random field.
This is in contrast to many settings in the literature where the data space is often finite-dimensional, and the noise is a random vector.

Under the assumption $\Gammaa \propto \mathrm{Id}_{\Ya}$, the random field $\etaa$ is a so-called \emph{white noise}, see \cite[pp.~40--41]{Stein1999} or \cite[pp.~7--9]{Kuo1996} for a definition and review of Gaussian white noise. 
Note that white noise is not a classical Gaussian random field, since the covariance operator is not trace-class in $\Ya  := L^2(\Omega)$, see for instance \cite[Cor.~2.3.2]{Bogachev1998}.
Instead one defines an \emph{abstract Wiener space}, that is a tuple $(\Ya, B)$, where $B$ is a separable Banach space containing $\Ya$.
The random field $\etaa$ takes values in $\Ya$, but those have to be tested with elements in $B$.  For instance, this can be specified by the characteristic function of $\etaa$ fulfilling
\begin{align*}
\int_B \exp\left(\mathrm{i}\langle x, y \rangle_{\Ya}\right)\mathbb{P}(\etaa \in \mathrm{d}x) = \exp\left(-\frac{1}{2\sigmaa^2}\| y\|^2\right) \text{ for } y \in \Ya.
\end{align*}
Due to this implicit definition $\etaa$ is referred to as \emph{generalised random field}. 
Furthermore, we assume that $\etaa$ is independent of the parameter $u$.  
In this case, the data-generating measure is given by $\mathrm{N}(\Ga(u^\dagger), \Gammaa)$, and the conditional probability measure of $y$ given $u$ is $\mathrm{N}(\Ga(u), \Gammaa)$. 

Let $\mathcal{B}\Ya$ denote the Borel-$\sigma$-algebra on $\Ya$.
By definition, the likelihood is a PDF of $\mathrm{N}(\Ga(u), \Gammaa)$ with respect to some measure on $(\Ya, \mathcal{B}\Ya)$ that does not depend on $u$.
We can construct such a PDF by applying the Cameron--Martin theorem.
It states that $\mathrm{N}(\Ga(u), \Gammaa)$ and $\mathrm{N}(0,\Gammaa)$ are
\emph{equivalent}, if $\Ga(X)$ is a subset of the so-called \emph{Cameron--Martin space} associated with $\mathrm{N}(0,\Gammaa)$.
If $\mathrm{N}(\Ga(u), \Gammaa)$ and $\mathrm{N}(0,\Gammaa)$ are equivalent, then the
likelihood can be defined by the Radon--Nikodym derivative of $\mathrm{N}(\Ga(u), \Gammaa)$ w.r.t.~$\mathrm{N}(0,\Gammaa)$.
Since $\Ya=L^2(\Omega)$ is a Hilbert space, the Cameron--Martin space
$\mathrm{CM}({\mathrm{N}(0,\Gammaa)})$ of $\mathrm{N}(0,\Gammaa)$ is the closure of $\mathrm{Img}(\Gammaa^{1/2};\Ya)$ with respect to the norm induced by the inner product $$\langle f, g \rangle_{\mathrm{CM}(\mathrm{N}(0,\Gammaa))}:= \langle \Gammaa^{-1/2}f,\Gammaa^{-1/2} g \rangle_{\Ya},$$ see \cite[p.~44]{Bogachev1998} and \cite[Def.~2.50]{Sullivan2015}.
The Cameron\change{--}Martin theorem for white noise is given in \cite[p.~8]{Kuo1996}.
In our setting it holds that $\mathrm{CM}({\mathrm{N}(0,\Gammaa)}) = \Ya = L^2(\Omega)$. 
Moreover, it follows from \Cref{lem:well-posed} below that $\Ga(u) = \varphi(\cdot, T) \in L^2(\Omega)$ for any $u \in X$. 
Hence $\Ga(X) \subseteq \mathrm{CM}({\mathrm{N}(0,\Gammaa)})$, and thus the measures $\mathrm{N}(\Ga(u), \Gammaa)$ and $\mathrm{N}(0,\Gammaa)$ are equivalent.
We obtain the following Radon\change{--}Nikodym derivative that we will use as the likelihood:

\begin{align*}
\frac{\mathrm{d}\mathrm{N}(\Ga(u),\Gammaa)}{\mathrm{d}\mathrm{N}(0,\Gammaa)} (\ya)= \exp\left(-\frac{\normYa{\Gammaa^{-1/2}\Ga(u)}^2}{2} + \langle \Ga(u), \Gammaa^{-1}\ya\rangle_{\Ya}\right).
\end{align*}
After a few simplifications we arrive at the following likelihood $\La$ and potential $\Phia$:
\begin{align}\label{likelihood}
\La(\ya|u) &:= \exp \left ( - \Phia(u;\ya) \right ), \\  \Phia(u;\ya) &:= \frac{\normYa{\Ga(u)}^2-2\langle \Ga(u), \ya\rangle_{\Ya}}{2\sigmaa^{2}} := \frac{\normYa{\varphi(\cdot,T)}^2-2\langle \varphi(\cdot,T), \ya\rangle_{\Ya}}{2\sigmaa^{2}}. \nonumber
\end{align}
 
\begin{remark} \label{rmk:Misfit}
In \change{the setting} where the data space is finite-dimensional and the noise is Gaussian, one usually considers a potential of the form
\begin{align*}
\hat{\Phi}_a(u;\ya) = \frac{1}{2}\normYa{\Gammaa^{-1/2}(\Ga(u) - \ya)}^2
\end{align*}
instead of $\Phia$ in \eqref{likelihood}.
However, this choice would not induce a correct data-generating measure in cases where $\Ya$ is infinite-dimensional, see  \cite[Rmk.~3.8]{Stuart} for more details.
In a finite-dimensional setting it can be shown that potentials of the forms $\hat{\Phi}_a$ and $\Phia$ lead to an equivalent Bayesian analysis.  
\end{remark}

For setting $(b)$ the forward response operator $\Gb$ maps from $X$ to $\Yb := \R^J$. 
The conditional distribution of $\yb$ given $u$ is the Gaussian measure $\mathrm{N}(\Gb(u), \Gammab)$.
We define the noise covariance by $\Gammab = \mathrm{diag}(\sigmab^2(t_j) : j = 1,\dots,J)$. 
Here, $\sigmab^2(t_j)$ is the noise variance of the measurement at time $t_j$, $j=1,\dots,J$. 
The noises at different points in time are stochastically independent.
Since the data space $\Yb$ is finite-dimensional, we can define the likelihood using the PDF of the multivariate Gaussian measure w.r.t.~the Lebesgue measure, which has also been employed in \cite[\S 3.2.2]{Collis_tutorial}.
However, for reasons of consistency and brevity in the following discussion, we again consider the PDF of $\yb$ given $u$ w.r.t.~the probability measure of the noise.
This PDF is well-defined since the noise covariance matrix $\Gammab$ is invertible and $\Yb$ is finite-dimensional. 
We arrive at
\begin{align}\label{likelihood_Collis}
\Lb(\yb|u) &:= \exp (-\Phib(u;\yb) ),\\  \Phib(u;\yb) &:= \frac{1}{2}\normYb{\Gammab^{-1/2}\Gb(u)}^2-\langle \Gb(u),\Gammab^{-1} \yb\rangle_{\Yb} := \sum_{j=1}^J\frac{\lb_j(\varphi)^2 - 2\yb_j\lb_j(\varphi)}{2\sigmab^2(t_j)}.
\end{align}
The likelihood in \eqref{likelihood_Collis} and the likelihood in \cite[\S 3.2.2]{Collis_tutorial} lead to identical estimation results (cf.~\Cref{rmk:Misfit}).
Finally, we note that when we discuss either of the cases $(a)$ or $(b)$ we drop the subscripts and write
\begin{align*}
L(y|u) &:= \exp \left ( - \Phi(u;y) \right ) \quad \text{ where } \quad  \Phi(u;y) := \frac{1}{2}\normY{\Gamma^{-1/2}\G(u)}^2-\langle \G(u),\Gamma^{-1} y\rangle_{Y}.
\end{align*}

\subsection{Properties of the forward model}\label{sec:fm}
We use the notation $L^p(\Omega)$ and $W^{k,p}(\Omega)$ for $p \in [1,\infty]$, $k > 0$, to denote the standard Lebesgue and Sobolev spaces.  
For $p = 2$ we use $H^k(\Omega) := W^{k,2}(\Omega)$, and we may drop the dependence on $\Omega$ when there is no ambiguity. 
The dual space of a Banach space $Z$ is denoted by $Z'$, and the duality pairing between $Z$ and $Z'$ is denoted by $\inner{\cdot}{\cdot}_{Z}$.  
The Bochner space $L^p(0,T;Z)$ for any $p \in [1,\infty]$ may be denoted as $L^p(Z)$.  
Due to the Neumann boundary conditions, we introduce the function space
\begin{align*}
H^2_n(\Omega) := \{ f \in H^2(\Omega) \, : \,  \pdnu f  = 0 \text{ on } \pd \Omega \}.
\end{align*}
For fixed positive constants $\PP_{\infty}$, $\chi_{\infty}$ and $\CC_{\infty}$ we define the parameter space $X$ as
\begin{align}\label{defn:X}
X := [0, \PP_{\infty}] \times [0,\chi_{\infty}] \times [0,\CC_{\infty}].
\end{align}
Now we state some useful properties of the forward solution operator $u \mapsto (\varphi, \mu, \sigma)$ obtained under the following assumptions:
\begin{enumerate}[label=$(\mathrm{A \arabic*})$, ref = $\mathrm{A \arabic*}$]
\item \label{ass:para} The parameters $(\PP, \chi, \CC) \in X$ and are constant in space and in time.
\item \label{ass:hfg} The functions $h, f$ and $g$ belonging to $C^2(\R)$ are bounded with bounded derivatives, and $h(s) \geq 0$ for all $s \in \R$.
\end{enumerate}
together with either
\begin{enumerate}[label=$(\mathrm{A3})$,ref = $\mathrm{A3}$]
\item \label{ass:C4} 
\begin{enumerate}[label=$(\mathrm{\roman*})$]
\item $\Omega \subset \R^{2}$ is a bounded domain with $C^{4}$-boundary $\pd \Omega$.
\item The mobility $m \in C^2(\R)$ is bounded with bounded derivatives and there exist $n_0$, $n_1 > 0$ such that $n_0 \leq m(s) \leq n_1$ for all $s \in \R$.
\item The initial conditions $\varphi_{0}, \sigma_{0}$ belong to $H^{3}(\Omega) \cap H^{2}_{n}(\Omega)$.
\item The potential $\Psi \in C^{3}(\R)$ is non-negative and there exist positive constants $R_{1}, \dots, R_{5}$ such that for all $s,t \in \R$, $l = 3$, and $k  \in \{1,2,3\}$,
\begin{equation}\label{Psi:prop}
\begin{aligned}
\Psi(s) \geq R_{1} \abs{s}^{2} - R_{2}, \, \abs{\Psi'(s)}  & \leq R_{3}(1 + \Psi(s)), \, |\Psi^{(l)}(s)| \leq R_{4}(1+\abs{s}^{q}), \\
|\Psi^{(k)}(s) - \Psi^{(k)}(t)| &  \leq R_{5} (1 + \abs{s}^{r-k+1} + \abs{t}^{r-k+1}) \abs{s - t}
\end{aligned}
\end{equation}
for some exponents $q \in [1,\infty)$ and $r \in [2,\infty)$.
\end{enumerate}
\end{enumerate}
or
\begin{enumerate}[label=$(\mathrm{A4})$, ref = $\mathrm{A4}$]
\item \label{ass:poly} 
\begin{enumerate}[label=$(\mathrm{\roman*})$]
\item $\Omega \subset \R^2$ is a convex domain with polygonal boundary.
\item The mobility $m$ is constant (w.l.o.g.~$m = 1$).
\item The initial conditions $\varphi_{0}, \sigma_{0}$ belong to $H^{2}_n(\Omega)$.
\item The potential $\Psi \in C^2(\R)$ is non-negative and there exist positive constants $R_1, \dots, R_5$ such that for all $s, t \in \R$, $l = 2$, and $k \in \{1,2\}$, the property \eqref{Psi:prop} holds 
for some exponents $q \in [1,\infty)$ and $r \in [1,\infty)$.
\end{enumerate}
\end{enumerate}


\change{
Note that assumption (A1) is biologically reasonable since the PDE system models a relatively young tumour where nutrients are plentiful.  
Hence we expect that all tumour cells evolve in the same manner, leading to constant rates of proliferation and nutrient consumption.  
Regarding the assumptions on the boundary of the domain, we have in mind the situation where tumour growth data can be obtained from medical imaging (square photographs that have Lipschitz boundaries - (A4i)) or from experiments (petri dishes that have smooth boundaries - (A3i)).  
The remaining assumptions (A2), (A3ii--iv) and (A4ii--iv) are technical assumptions, and are required in the forthcoming proofs.
}

We introduce the following notation for the case where $\Omega$ is a $C^4$-domain and the mobility $m$ depends on $\varphi$:
\begin{align*}
\AA_1 & := L^{\infty}(H^{3} ) \cap L^{2}(H^{4}) \cap H^{1}(L^{2}), \; \AA_2  := L^{\infty}(H^{1}) \cap L^{2}(H^{3}), \; \AA_3 := L^{\infty}(H^{3}) \cap H^{1}(H^{2}_n), \\
\BB_{1}  & := L^{\infty}(H^{2}_n) \cap L^{2}(H^{4}) \cap H^{1}((H^{1})'), \;  \BB_{2} := L^{\infty}(L^{2}) \cap L^{2}(H^{2}_n), \;
\BB_{3} := L^{\infty}(H^{1}) \cap L^{2}(H^{2}_n),
\end{align*}
and correspondingly the following notation for the case where $\Omega$ is a polygonal domain and the mobility $m = 1$ is constant:
\begin{align*}
\WW_1 &:= L^{\infty}(H^2_n) \cap H^1(L^2), \; \WW_2 :=  L^{\infty}(L^2) \cap L^2(H^2_n) \cap H^1((H^2_n)'), \; \WW_{3} :=  H^1(H^2_n), \\
 \Z_{1}  & := L^{\infty}(H^2_n) \cap H^1(L^2), \;  \Z_{2}  := L^{\infty}(L^2) \cap L^2(H^2_n), \; \Z_{3} := L^{\infty}(H^1) \cap L^2(H^2_n).
\end{align*}
The strong well-posedness of \eqref{CH} is formulated as follows.
\begin{lemma}[Strong well-posedness]
\label{lem:well-posed}
Under assumptions \eqref{ass:para}, \eqref{ass:hfg} and \eqref{ass:C4}, for any $u = (\PP, \chi, \CC) \in X$ and any fixed but arbitrary constant $T > 0$, there exists a unique triplet of functions $(\varphi, \mu, \sigma) \in \AA_1 \times \AA_2 \times \AA_3$ that is a strong solution of \eqref{CH} satisfying the initial and boundary conditions, and the estimate
\begin{align}\label{bound:2}
\norm{\varphi}_{\AA_{1}} + \norm{\mu}_{\AA_{2}} + \norm{\sigma}_{\AA_{3}} \leq C_{*} (1 + \abs{u}),
\end{align}
for some positive constant $C_*$ not depending on $(\varphi, \mu, \sigma)$.  Furthermore, let $\{(\varphi_i, \mu_i, \sigma_i)\}_{i = 1,2}$ denote two strong solutions of \eqref{CH} corresponding to parameters $\{\varphi_{0,i}, \sigma_{0,i}, \PP_{i}, \chi_{i}, \CC_{i}\}_{i=1,2}$ but with the same initial data.  Then, there exists a positive constant $C^* > 0$ not depending on the differences $\varphi_{1} - \varphi_{2}$, $\mu_{1} - \mu_{2}$, $\sigma_{1} - \sigma_{2}$, $\PP_{1} - \PP_{2}$, $\chi_{1} - \chi_{2}$ and $\CC_{1} - \CC_{2}$ such that
\begin{equation}\label{ctsdep}
\begin{aligned}
& \norm{\varphi_{1} - \varphi_{2}}_{\BB_{1}} + \norm{\mu_{1} - \mu_{2}}_{\BB_{2}} + \norm{\sigma_{1} - \sigma_{2}}_{\BB_{3}} \\
& \quad \leq C^* \left (\abs{\PP_{1} - \PP_{2}} + \abs{\chi_{1} - \chi_{2}} + \abs{\CC_{1} - \CC_{2}} \right ).
\end{aligned}
\end{equation}
Analogously, if \eqref{ass:C4} is replaced by \eqref{ass:poly}, then we replace $\AA_i$ with $\WW_i$ and $\BB_i$ with $\Z_i$ for $i = 1,2,3$.
\end{lemma}

\begin{proof}
For the details where \eqref{ass:C4} is assumed, i.e., $\Omega$ is a $C^4$-domain and the mobility $m$ need not be a constant, we refer the reader to the earlier work \cite{KL}.  We note that the corresponding assertions for the case where \eqref{ass:poly} is assumed, i.e., $\Omega$ is a convex polygonal domain and the mobility is taken to be constant, is actually an improvement to \cite[Thm.~8]{KL} with new continuous dependence results for $\varphi$ in $L^{\infty}(0,T;H^2(\Omega)) \cap H^1(0,T;L^2(\Omega))$ and for $\mu$ in $L^{\infty}(0,T;L^2(\Omega)) \cap L^2(0,T;H^2(\Omega))$.  These new results are needed for problem setting $(a)$, cf.~\Cref{rem:GaGbreg} below, and we only sketch the new details.

The new regularity result $\mu \in H^1(0,T;H^2_n(\Omega)')$ can be obtained by differentiating \eqref{mu} in time at the Galerkin level, testing with an arbitrary test function $\zeta \in H^2_n(\Omega)$, and employing previously established estimates for $\varphi$ in $L^{\infty}(0,T;H^2(\Omega)) \cap H^1(0,T;L^2(\Omega))$ from \cite[(56)]{KL}, and then passing to the limit in the Galerkin approximation.

For the new continuous dependence result, we use the notation $\theta_i = \theta(\varphi_i)$ for $\theta \in \{h, f, g, \Psi'\}, i = 1,2$, and $\hat z = z_1- z_2$ for $z \in \{ \varphi, \mu, \sigma, \PP, \chi, \CC, f, h, g, \Psi'\}$, and set $\mathcal{Y} := |\widehat \chi|^2 + |\widehat \PP|^2 + |\widehat \CC|^2$.  Then, our starting point is the following estimate obtained as a consequence of \cite[(11)-(14),(18)]{KL}:
\begin{align}\label{ctsLipdom}
\change{\norm{\hat \varphi}_{L^{\infty}(0,T;H^1(\Omega))}^2 + \norm{\hat \varphi}_{H^1(0,T;H^1(\Omega)')}^2 + \norm{\hat \mu}_{L^2(0,T;H^1(\Omega))}^2 + \norm{\hat \sigma}_{\Z_{3}}^2} \leq C \mathcal{Y}
\end{align}
for some positive constant $C$ not depending on the differences $\hat \varphi$, $\hat \mu$, $\hat \sigma$, $\widehat \PP$, $\widehat \chi$ and $\widehat \CC$.  To obtain the new continuous dependence estimates, let us note that thanks to the embedding $H^2(\Omega) \subset L^{\infty}(\Omega) \cap W^{1,4}(\Omega)$, we have $\varphi_{1}, \varphi_{2} \in L^{\infty}(Q)$ and $\nabla \varphi_2 \in L^{\infty}(0,T;L^4(\Omega))$, leading to
\begin{align*}
\norm{\hat \Psi'}_{L^{\infty}(0,T;L^2(\Omega))}^2 & \leq C \Big ( 1 + \norm{\varphi_1}_{L^{\infty}(Q)}^{2r} + \norm{\varphi_2}_{L^\infty(Q)}^{2r} \Big ) \norm{\hat \varphi}_{L^{\infty}(0,T;L^2(\Omega))}^2 \leq C \mathcal{Y}, \\
\norm{\nabla \hat \Psi'}_{L^{\infty}(0,T;L^2(\Omega))}^2 & \leq C \int_\Omega \abs{\Psi''(\varphi_1)}^2 \abs{\nabla \hat \varphi}^2 + \abs{\nabla \varphi_2}^2 \abs{\Psi''(\varphi_1) - \Psi''(\varphi_2)}^2 \dx \\
& \leq C \norm{\nabla \hat \varphi}_{L^{\infty}(0,T;L^2(\Omega))}^2 + C\norm{\nabla \varphi_2}_{L^{\infty}(0,T;L^4(\Omega))}^2 \norm{\hat \varphi}_{L^{\infty}(0,T;L^4(\Omega))}^2 \leq C \mathcal{Y}.
\end{align*}
From the equations fulfilled by the differences
\begin{subequations}
\begin{alignat}{3}
\hat \varphi_{t} & = \Lap \hat \mu - \widehat \chi \Lap \sigma_1 - \chi_2 \Lap \hat \sigma + \widehat \PP f_1 g_1 + \PP_2 \hat f g_1 + \PP_2 f_2 \hat g, \label{diff:varphi} \\
\hat \mu & = \hat \Psi' - \Lap \hat \varphi, \label{diff:mu}
\end{alignat}
\end{subequations}
we apply elliptic regularity to \eqref{diff:mu} to deduce that 
\begin{align}\label{ctsH2}
\norm{\hat \varphi}_{L^2(0,T;H^2(\Omega))}^2 \leq C \Big ( \norm{\hat \mu - \hat \Psi'}_{L^2(0,T;L^2(\Omega))}^2 + \norm{\hat \varphi}_{L^2(0,T;H^1(\Omega))}^2 \Big ) \leq C \mathcal{Y}.
\end{align}
Next, differentiating \eqref{diff:mu} in time and testing with $\hat \mu$ (possible thanks to the fact that $\hat \mu \in L^2(0,T;H^2_n(\Omega)) \cap H^1(0,T;H^2_n(\Omega)')$) leads to
\begin{align*}
\frac{1}{2} \frac{d}{dt} \norm{\hat \mu}_{L^{2}(\Omega)}^2 & = \inner{\hat \mu_t}{\hat \mu}_{H^2_n} = \int_{\Omega} (\Psi''(\varphi_1) - \Psi''(\varphi_2)) \varphi_{1,t} \hat \mu + \Psi''(\varphi_2) \hat \varphi_t \hat \mu - \hat \varphi_{t} \Lap \hat \mu \dx \\
& \leq C \Big ( \norm{\hat \varphi}_{L^{\infty}(\Omega)} \norm{\varphi_{1,t}}_{L^2(\Omega)} + \norm{\hat \varphi_t}_{L^2(\Omega)} \Big )\norm{\hat \mu}_{L^2(\Omega)} - \int_\Omega \hat \varphi_t \Lap \hat \mu \dx,
\end{align*}
where we used that $\varphi_1, \varphi_2, \Psi''(\varphi_2) \in L^{\infty}(Q)$.  Testing \eqref{diff:varphi} with $\hat \varphi_{t}$ yields
\begin{align*}
\norm{\hat \varphi_{t}}_{L^{2}(\Omega)}^2 & = \int_{\Omega} \hat \varphi_{t} \Big ( \Lap \hat \mu - \widehat \chi \Lap \sigma_{1} - \chi_{2} \Lap \hat \sigma + \widehat P f_1 g_1 + \PP_2 \hat f g_2 + \PP_2 f_2 \hat g \Big ) \dx \\
& \leq \int_\Omega \hat \varphi_t \Lap \hat \mu \dx + C \Big (\abs{\widehat \chi}^2 + \norm{\hat \sigma}_{H^2(\Omega)}^2 + |\widehat \PP|^2 + \norm{\hat \varphi}_{L^2(\Omega)}^2 \Big ) + \frac{1}{2} \norm{\hat \varphi_t}_{L^2(\Omega)}^2,
\end{align*}
where we used the boundedness and Lipschitz continuity of $f$ and $g$ from \eqref{ass:hfg}, as well as the fact that $\sigma_1 \in \WW_3$.  Then, adding these two inequalities to cancel the common term involving $\hat \varphi_t \Lap \hat \mu$, applying Gronwall's inequality, \eqref{ctsLipdom} and \eqref{ctsH2} leads to
\begin{align*}
\norm{\hat \mu}_{L^{\infty}(0,T;L^2(\Omega))}^2 + \norm{\hat \varphi}_{H^1(0,T;L^2(\Omega))}^2 \leq C \mathcal{Y},
\end{align*}
and in turn, recalling $\norm{\hat \Psi'}_{L^{\infty}(0,T;L^2(\Omega))}^2 \leq C \mathcal{Y}$ we obtain via elliptic regularity
\begin{align*}
\norm{\hat \varphi}_{L^{\infty}(0,T;H^2(\Omega))}^{\change{2}} \leq C \mathcal{Y}.
\end{align*}
The assertion $\norm{\hat \mu}_{L^2(0,T;H^2(\Omega))}^2 \leq C \mathcal{Y}$ follows from elliptic regularity applied to 
\begin{align*}
-\Lap \hat \mu = - \hat \varphi_t - \widehat \chi \Lap \sigma_1 - \chi_2 \Lap \hat \sigma + \widehat \PP f_1 g_1 + \PP_2 \hat f g_1 + \PP_2 f_2 \hat g,
\end{align*}
where the right-hand side in the $L^2(0,T;L^2(\Omega))$-norm is bounded by $C \mathcal{Y}^{1/2}$.  This completes the proof.
\end{proof}

The interesting consequence for the Bayesian inverse problem \eqref{obs:data:com} is the following.

\begin{corollary}\label{cor:G:prop}
Under \eqref{ass:para}, \eqref{ass:hfg} and either \eqref{ass:C4} or \eqref{ass:poly}, for both settings $(a)$ and $(b)$, there exist constants $E, F > 0$ such that the forward solution operator $\G : X \to Y$ satisfies
\begin{align}
\normY{\G(u)} & \leq  E( 1 + \abs{u}) \label{G:upp}, \\
\normY{\G(u_{1})- \G(u_{2})} & \leq F \abs{u_{1} - u_{2}} \label{G:diff}
\end{align}
for $u, u_1, u_2 \in X$.
\end{corollary}
\begin{proof}
For setting $(a)$, using the (compact) embedding 
\begin{align}\label{embedding}
L^{\infty}(0,T;H^{2}(\Omega)) \cap H^{1}(0,T;H^{1}(\Omega)') \subset \subset C^{0}([0,T];C^{0}(\overline{\Omega}))
\end{align}
yields that $\varphi$ belongs to $C^{0}([0,T];C^{0}(\overline{\Omega}))$.  Hence, for our choice $l_{1}(\varphi) = \varphi(\cdot, T)$, which is a function on $\Omega$, or $\{l_{j}(\varphi)\}_{j=1}^{J} = \{\frac{1}{2}\int_{\Omega}1+\varphi(x, t_{j})\mathrm{d}x\}_{j=1}^{J}$, where $\{t_{j}\}_{j=1}^{J}$ is a discrete set of points in $[0,T]$, we infer from \eqref{bound:2} that
\begin{subequations}
\begin{alignat}{3}
\label{Ga} \normYa{\Ga(u)} &\leq \abs{\Omega}^{1/2} \norm{\varphi}_{C^{0}([0,T];C^{0}(\overline{\Omega}))} \leq \abs{\Omega}^{1/2} C_{*}(1 + \abs{u}) = E_a( 1 + \abs{u}),\\
\label{Gb} \normYb{\Gb(u)} &\leq C  (\norm{\varphi}_{C^{0}([0,T];L^{2}(\Omega))}+1) \leq C (C_{*}(1 + \abs{u}) + 1 ) \leq E_b( 1 + \abs{u}),
\end{alignat}
\end{subequations}
for some constants $E_a, E_b > 0$.  Taking $E := \max\{E_a, E_b\}$ leads to \eqref{G:upp}.

Similarly, from the embedding \eqref{embedding} and continuous dependence result \eqref{ctsdep} it holds
\begin{align}\label{phi:diff}
\norm{\varphi_{1} - \varphi_{2}}_{C^{0}([0,T];C^{0}(\overline{\Omega}))} \leq C^{*} \left (\abs{\PP_{1} - \PP_{2}} + \abs{\chi_{1} - \chi_{2}} + \abs{\CC_{1} - \CC_{2}} \right ).
\end{align}
Then, \eqref{G:diff} can be derived in a similar fashion to \eqref{G:upp}.
\end{proof}

\begin{remark}\label{rem:GaGbreg}
For setting $(a)$ the minimum regularity in the continuous dependence result is in the space $C^0([0,T];C^0(\overline{\Omega}))$, which is \eqref{phi:diff}.  However, for setting $(b)$ it is sufficient to have a weaker continuous dependence result in $C^{0}([0,T];L^2(\Omega))$, compare \eqref{Gb}.
\end{remark}

\begin{remark}
To establish the weaker continuous dependence result in $C^0([0,T];L^2(\Omega))$ for problem setting $(b)$ in the presence of a non-constant mobility $m(\varphi)$ comes at a cost of deriving high solution regularities $\varphi \in L^{\infty}(0,T;W^{1,\infty}(\Omega))$ and $\mu \in L^{4}(0,T;W^{1,4}(\Omega))$, which can only be achieved at present with a $C^4$-boundary, see the derivation of \cite[(11)]{KL} for more details.  
\end{remark}

\subsection{Properties of the potential}
The following result shows that the negative log-likelihood $\Phi(u;y)$
defined in \eqref{likelihood} satisfies the assumptions outlined in \cite[Assump.~2.6]{Stuart}, which are
\begin{enumerate}[label=$(\mathrm{B \arabic*})$, ref = $\mathrm{B \arabic*}$]
\item \label{lem:1} For every $\eps > 0$ and $r > 0$, there exists a constant $M = M(\eps, r) \in
\R$ such that for all $u \in X$ and for all $y \in Y$ with $\normY{y} < r$ it
holds $\Phi(u;y) \geq M - \eps \abs{u}^{2}$.
\item \label{lem:2} For every $r > 0$, 
there exists $K = K(r) > 0$ such that for all $u \in X$ and $y \in Y$ with
$\max(\abs{u}, \normY{y}) < r$ it holds $\Phi(u;y) \leq K$.
\item \label{lem:3} For every $r > 0$, there exists an $L = L(r) > 0$ such that for all $u_{1}, u_{2} \in X$ and $y \in Y$ with $\max ( \abs{u_{1}} , \abs{u_{2}}, \normY{y}) < r$, it holds $\abs{\Phi(u_{1};y) - \Phi(u_{2};y)} \leq L \abs{u_{1} - u_{2}}$.
\item \label{lem:4} For every $\eps > 0$ and $r > 0$, there exists $H = H(\eps,r) \in \R$ such that for all $y_{1}, y_{2} \in Y$ with $\max(\normY{y_{1}}, \normY{y_{2}}) < r$ and for every $u \in X$, it holds
$\abs{\Phi(u;y_{1}) - \Phi(u;y_{2})} \leq \exp(\eps \abs{u}^{2} + H) \normY{y_{1} - y_{2}}$.
\end{enumerate}

This is an important step in the proof of the well-posedness of the Bayesian inverse problem, and is required to prove the existence and uniqueness of the posterior measure.

\begin{lemma}
Under \eqref{ass:para}, \eqref{ass:hfg} and either \eqref{ass:C4} or \eqref{ass:poly}, the properties \eqref{lem:1}, \eqref{lem:2}, \eqref{lem:3} and \eqref{lem:4} are satisfied.
\end{lemma}

\begin{proof} For convenience, we define $u_\infty := (\PP_\infty, \chi_\infty, \CC_\infty)$ and
\begin{align*}
q := \frac{1}{\max\{ \sigma_a^2, \sigma_b^2(t_1), \dots, \sigma_b^2(t_J)\}}, \quad q' := \frac{1}{\min \{  \sigma_a^2, \sigma_b^2(t_1), \dots, \sigma_b^2(t_J) \}}.
\end{align*}
For \eqref{lem:1}, we observe that $\Gamma^{-1}$ is a strictly positive definite operator with its spectrum bounded below by $q > 0$.  Hence, for any $\xi \in Y$,  it holds that $\langle\xi^{*},\Gamma^{-1} \xi \rangle_Y \geq q \normY{\xi}^{2}$.  Then by \eqref{G:upp}, for $\norm{y}_{Y} < r$ and arbitrary $\eps > 0$, setting $M := - r q E (1 + \abs{u_\infty})$, where $E$ is the constant in \eqref{G:upp}, we have
\begin{align*}
\Phi(u;y) &= \frac{1}{2} \normY{\Gamma^{-1/2}\mathcal{G}(u)}^2 - \langle\G(u),\Gamma^{-1}y \rangle \geq - q\normY{\G(u)}\normY{y} \\
&\geq -rqE(1+\abs{u}) \geq -rqE(1+\abs{u_\infty}) = M \geq M - \eps \abs{u}^{2}.
\end{align*}
For \eqref{lem:2}, note that the spectrum of $\Gamma^{-1}$ is bounded above by $q' < \infty$, we use \eqref{G:upp} to see that
\begin{align*}
\Phi(u;y) &= \frac{1}{2} \normY{\Gamma^{-1/2}\mathcal{G}(u)}^2 - \langle\G(u),\Gamma^{-1}y \rangle \leq \normY{\Gamma^{-1/2}\mathcal{G}(u)}^2 + \frac{1}{2}\normY{\Gamma^{-1/2}y}^2 \\
&\leq q' \normY{\mathcal{G}(u)}^2 + \frac{q'}{2}\normY{y}^2 \leq q' E^2(1+\abs{u})^2 + \frac{q'}{2}\normY{y}^2  \leq q'\left(E^2(1+r)^2 + r^2 \right) =:K(r).
\end{align*}
For \eqref{lem:3}, a short computation shows that
\begin{equation}\label{diff}
\begin{aligned}
2\left (\Phi(u_{1};y) - \Phi(u_{2};y) \right ) 
& = \normY{\Gamma^{-1/2}\mathcal{G}(u_1)}^2 - 2\langle \G(u_1)-\G(u_2),\Gamma^{-1}y \rangle - \normY{\Gamma^{-1/2}\mathcal{G}(u_2)}^2 \\
& = \langle \G(u_1)-\G(u_2), \Gamma^{-1} (\G(u_1)- 2  y + \G(u_2))\rangle_Y .
\end{aligned}
\end{equation}
Using \eqref{G:upp}, \eqref{G:diff}, the Cauchy--Schwarz and triangle inequalities, we arrive at
\begin{align*}
\abs{\Phi(u_{1};y) - \Phi(u_{2};y)} &\leq \frac{1}{2}\normY{\G(u_1)-\G(u_2)}\normY{\Gamma^{-1}(\G(u_1)+\G(u_2)-2y)} \\ 
&\leq F \abs{u_1 - u_2}  q' (E(1+r)+r) =: L(r) \abs{u_{1} - u_{2}}.
\end{align*}
For \eqref{lem:4}, we obtain after a short computation
\begin{align*}
2 (\Phi(u;y_{1}) - \Phi(u;y_{2})) = \langle y_2-y_1, \Gamma^{-1}(2\G(u)-y_1-y_2)\rangle_Y.
\end{align*}
Then, due to $\max(\normY{y_{1}}, \normY{y_{2}}) < r$ and \eqref{G:upp}, we see that
\begin{align*}
\abs{\Phi(u;y_{1}) - \Phi(u;y_{2})} & \leq q' ( E(1 + \abs{u}) + r) \normY{y_{1} - y_{2}} \leq (\eps \abs{u}^2 + C) \normY{y_{1} - y_{2}} \\
& \leq \exp(\eps \abs{u}^{2}+C) \normY{y_{1} - y_{2}},
\end{align*}
for some positive constant $C=C(\eps, q', E, r)$.  This finishes the proof.
\end{proof}

\section{Bayesian inversion}\label{sec:Bayesian}
Recall that the parameter vector $u=(\mathcal{P},\chi,\mathcal{C}) \in X \subset \mathbb{R}^3$ is contained in a finite-dimensional space.
Let $\mu_{0}$ denote a prior probability measure for $u$ with corresponding probability density function $\pi_{0}$.  
Our interest is the posterior probability measure of $u$ given $y$ which we denote by $\mu^{y}$ with probability density function $\pi^{y}$.  
By Bayes' rule we find the Radon--Nikodym relation
\begin{align}\label{RN}
\frac{\mathrm{d}\mu^{y}}{\mathrm{d} \mu_{0}}(u) =  \frac{1}{Z(y)}\exp(-\Phi(u;y)), \quad Z(y) := \int_{X} \exp(-\Phi(u;y)) \, \mathrm{d} \mu_{0}(u).
\end{align}
To be able to apply the Bayesian framework for inverse problems developed in \cite{Cotter, Stuart} we require a prior measure $\mu_{0}$ that satisfies $\mu_{0}(X) = 1$. This means that the functions drawn from the prior measure $\mu_{0}$ belong to the space $X$ almost surely.  

For the benefit of the reader we recall that for two probability measures $\mu_{1}$ and $\mu_{2}$ on a measurable space $(X, \mathcal{B}X)$, both absolutely continuous with respect to the same $\sigma$-finite reference measure $\nu$, i.e., $\mu_{i} \ll \nu$ for $i = 1,2$, the Hellinger distance $d_{\mathrm{H}}(\mu_1, \mu_2)$ between $\mu_{1}$ and $\mu_{2}$ is defined as
\begin{align*}
d_{\mathrm{H}}(\mu_{1}, \mu_{2}) = \left ( \frac{1}{2} \int_{X} \left ( \sqrt{\frac{\mathrm{d} \mu_{1}}{\mathrm{d} \nu}} - \sqrt{\frac{\mathrm{d} \mu_{2}}{\mathrm{d} \nu}} \right )^{2} \mathrm{d} \nu \right )^{1/2}
\end{align*}
with Radon--Nikodym derivatives $\frac{\mathrm{d} \mu_{1}}{\mathrm{d} \nu}$ and $\frac{\mathrm{d} \mu_{2}}{\mathrm{d} \nu}$ of $\mu_{1}$ and $\mu_{2}$, respectively.

\subsection{Well-posedness}
The well-posedness result for the Bayesian inverse problem \eqref{obs:data:com} is formulated as follows.

\begin{theorem}\label{thm:wellposed}
Consider the inverse problem of finding parameters $u = (\PP, \chi, \CC)$ from noisy observations of the form \eqref{obs:data} subject to $\varphi$ solving \eqref{CH}, with observational noise $\eta \sim \mathrm{N}(0,\Gamma)$ where $\Gamma$ is a strictly positive definite covariance operator.  Let $\mu_{0}$ be a prior measure satisfying
\begin{align}\label{ass:mu0}
\mu_{0}(X) = 1, \quad \mu_{0}(\{\abs{u} < r \} \cap X) > 0 \text{ for all } r > 0,
\end{align}
where $X$ is the space defined in \eqref{defn:X}.  Then, under \eqref{ass:para}, \eqref{ass:hfg} and either \eqref{ass:C4} or \eqref{ass:poly}, the posterior measure $\mu^{y}$ given by the relation \eqref{RN} is a well-defined probability measure and is Lipschitz continuous in the Hellinger metric with respect to the data, i.e., for any $r > 0$ there exists a positive constant $C = C(r)$ such that for all $y_{1}, y_{2} \in Y$ with $\max (\normY{y_{1}}, \normY{y_{2}}) < r$, it holds that $d_{\mathrm{H}}(\mu^{y_{1}}, \mu^{y_{2}}) \leq C \normY{y_{1} - y_{2}}$.
\end{theorem}

\begin{proof}
The estimate \eqref{G:diff} tells us that the forward response operator
$\G:X \to Y$ is Lipschitz continuous with respect to $u$.  Furthermore, by assumption $\mu_{0}(X) = 1$, it holds that $\G$ is $\mu_{0}$-almost surely continuous, and thus $\G$ is also $\mu_{0}$-measurable.  By \eqref{lem:3} the continuity of $\Phi(\cdot \, ;y)$ with respect to $u$ then implies that $\Phi(\cdot \, ;y)$ is also $\mu_{0}$-measurable.  Let $u_m = \sqrt{3} \max(\PP_{\infty}, \chi_{\infty}, \CC_{\infty})$ so that $\abs{u} \leq u_m$ for all $u \in X$.  Then, for any $\eps > 0$, we find that
\begin{align}
\label{norma:lb}
0 & \underset{\eqref{ass:mu0}}{<}\exp(-K(r)) \mu_{0}(\{\abs{u} < r \} \cap X) 
\underset{\eqref{lem:2}}{\leq} \int_{X} \exp(-\Phi(u;y)) \mathrm{d} \mu_{0}(u) = Z(y) \\
\notag & \underset{\eqref{lem:1}}{\leq} \int_{X} \exp(\eps \abs{u}^{2} - M(\eps,r)) \, \mathrm{d} \mu_{0}(u) \leq \exp(\eps u_m^{2} - M(\eps,r)) \mu_{0}(X) < \infty. 
\end{align}
Therefore, the measure $\mu^{y}$ defined via the relation \eqref{RN} is well-defined on $X$.  To show the Lipschitz dependence of $\mu^{y}$ on the data $y$, let $Z_i$ denote the normalization constants for $\mu^{y_i}$ for $i = 1,2$.  We take note of the useful identity
\begin{align}\label{diff:id}
e^{-\Phi(u; y_1)} - e^{-\Phi(u;y_2)} = \int_0^1 e^{-(z \Phi(u;y_1) + (1-z) \Phi(u;y_2))} \, \mathrm{d}z \; (\Phi(u;y_1) - \Phi(u;y_2)).
\end{align}
Then, for $\max (\normY{y_{1}}, \normY{y_{2}}) < r$, and any $\eps > 0$, applying \eqref{lem:1} and \eqref{lem:4} yields
\begin{align}\label{diff:normal}
\abs{Z_{1} - Z_{2}} 
 \leq \left (\int_{X} e^{\eps \abs{u}^{2} - M} e^{\eps \abs{u}^{2} + H} \, \mathrm{d} \mu_{0}(u) \right ) \normY{y_{1} - y_{2}} \leq C \normY{y_1 - y_2}.
\end{align}
From the definition of the Hellinger distance, using \eqref{lem:1}, \eqref{lem:4}, \eqref{norma:lb}, \eqref{diff:id} and \eqref{diff:normal}, and writing $Z_i^{-\frac{1}{2}} = Z_i / Z_i^{-\frac{3}{2}}$, we obtain
\begin{align*}
& 2 d_{\mathrm{H}}(\mu^{y_{1}}, \mu^{y_{2}})^{2} = \int_{X} \left ( \Big (Z_{1}^{-\frac{1}{2}} - Z_{2}^{-\frac{1}{2}} \Big ) e^{-\frac{1}{2} \Phi(u;y_{1})} + Z_{2}^{-\frac{1}{2}} \Big (e^{-\frac{1}{2} \Phi(u;y_{1})} - e^{-\frac{1}{2} \Phi(u;y_{2})} \Big )  \right)^{2} \mathrm{d} \mu_{0}(u) \\
& \quad \leq 2 \abs{Z_{1}^{-\frac{1}{2}} - Z_{2}^{-\frac{1}{2}}}^{2} \int_{X} e^{-\Phi(u;y_{1})} \mathrm{d} \mu_{0}(u) + \frac{2}{Z_{2}} \int_{X} \left (e^{-\frac{1}{2} \Phi(u;y_{1})} - e^{-\frac{1}{2} \Phi(u;y_{2})} \right)^{2} \mathrm{d} \mu_{0}(u) \\
& \quad \leq C  \abs{Z_1^{-\frac{1}{2}} - Z_2^{-\frac{1}{2}}}^2 + C \int_{X} \abs{\Phi(u;y_{1}) - \Phi(u;y_{2})}^{2} e^{2 (\eps \abs{u}^{2} - M(\eps,r))} \, \mathrm{d} \mu_{0}(u) \\
& \quad \leq C \max(Z_1^{-3},Z_2^{-3}) \abs{Z_1 - Z_2}^2 + C \int_{X} e^{2(\eps \abs{u}^{2}-M(\eps,r))} e^{2(\eps \abs{u}^{2} + H(\eps,r))} \, \mathrm{d} \mu_{0}(u) \normY{y_{1} - y_{2}}^{2} \\
 & \quad \leq C \normY{y_{1} - y_{2}}^{2}.
\end{align*}
This completes the proof.
\end{proof}

\subsection{Truncated Gaussian priors}
\label{ssec:priors} 
A natural example for a prior measure $\mu_{0}$ is the truncated Gaussian measure.  
This is a generalisation of a Gaussian probability measure on a bounded set.
To construct it we fix $(m_{\PP}, m_{\chi}, m_{\CC}) \in X$ and select positive constants $(\sigma_{\PP}, \sigma_{\chi}, \sigma_{\CC})$.  
Then we define
\begin{align*}
\phi(b,c;x) :=\exp \left ( - \frac{1}{2} \frac{(x-b)^{2}}{c^{2}} \right ), \quad \psi(a,b,c;x) := \frac{\phi(b,c;x)\Ind_{[0,a]}(x)}{\int_{0}^{a} \phi(b,c;x) \, \mathrm{d}x}.
\end{align*}
It is easy to see that $\psi$ is the density function of a truncated Gaussian on the interval $[0,a]$. 
The underlying untruncated Gaussian measure has mean $b$ and variance $c^{2}$.  
Our truncated Gaussian prior takes the form
\begin{align*}
\mu_{0}^{\mathrm{TG}}(A) = \int_{X \cap A} \prod_{i \in \{\PP,\chi, \CC\}} \psi(i_{\infty}, m_{i}, \sigma_{i}; x_{i}) \, \mathrm{d} x_{\PP}  \,\mathrm{d} x_{\chi} \, \mathrm{d} x_{\CC}
\end{align*}
for any measurable subset $A \subset \R^{3}$. 
The product structure of the prior means that we assume a priori that the parameters $\PP, \chi$ and $\CC$ are (stochastically) independent.
From the definition of $\psi$, it is clear that draws from $\mu_{0}^{\mathrm{TG}}$ lie in $X := [0, \PP_{\infty}] \times [0,\chi_{\infty}] \times [0,\CC_{\infty}]$ almost surely.
Moreover, the positivity of the density function $\pi_{0}$ guarantees that
$\mu_{0}^{\mathrm{TG}}(\{\abs{u} < r \} \cap X) > 0$ for all $r > 0$.
Hence, by \Cref{thm:wellposed}, the Bayesian inverse problem with the truncated Gaussian prior $\mu_{0}^{\mathrm{TG}}$ is well-posed.
 
\begin{remark}
An analogous well-posedness result can be proved for various absolutely continuous prior probability measures concentrated on $X$, such as a uniform distribution on $X$ or a truncated log-normal distribution.  
We do not investigate the choice of a prior measure for our Bayesian inverse problem in detail for the following reason:
due to the high-dimensional data space $Y$ and the small noise levels, we expect that the suggested likelihoods are highly informative. 
In this case, if \eqref{ass:mu0} holds, then the posterior measure can be considered independent of the prior measure.
This is a consequence of the asymptotic statement of the Bernstein--von Mises Theorem, see \cite[\S 10]{vdVaart1998} for a rigorous introduction or \cite[p.~568--569]{Owhadi2015} for an informal discussion.
\end{remark}

\begin{remark}
	\change{
	In the literature, many authors use uniform priors on non-negative bounded intervals, where the upper bound is often derived by a realistic guess from the data, see e.g. \cite{Oden2010}.
	However, Simpson et al. \cite{Simpson2017} point out  that a uniform prior might not always be a suitable choice.
	We think that the parameters should neither be too small nor too large, i.e., that they are likely located away from the endpoints of the given parameter intervals.
	This motivates our choice of a (truncated) Gaussian prior.}	
\end{remark}

\section{The fully discrete tumour model}
\label{sec:FE} 

Let $0=t_0 < t_1 < \ldots < t_k < \ldots < t_K = T$ denote a subdivision of the interval $I = [0,T]$. 
At time instance $t_k$ we define a subdivison $\mathcal T_h^k = \{T_i^k\}_{i=1}^{N^k}$ of $\overline \Omega$ containing 
closed triangles $T_i^k$ that exactly represent $\overline\Omega$, which in the following is assumed to be bounded with polygonal boundary.
On $\mathcal T_h^k$ we define the finite element function space 
\begin{align*}
  V_h^k = \{ v \in C(\overline\Omega)\,|\, v|_{T_i^k} \mbox{ is linear}, \, i=1,\ldots,N^k\},
\end{align*}
i.e., the space of piecewise linear and globally continuous finite element functions.  
We point out that from the dynamics associated with the Cahn--Hilliard equation, it is expected that the variable $\varphi$ takes constant values in large regions of $\Omega$ and is rapidly changing near the growing front of the tumour.  
Thus, adaptive meshing is necessary and therefore we use a different subdivision of the spatial domain at every time step.

At time instance $k$ we construct finite element approximations  $\varphi_h^k$, $\mu_h^k$, $\sigma_h^k$ $\in V_h^k$
of $\varphi$, $\mu$, $\sigma$, respectively.
To this end let $\varphi^{k-1}, \sigma^{k-1} \in V^{k-1}_h$ be given, 
let $I_h^k : C(\overline \Omega) \to V^k_h$ denote the Lagrangian interpolation operator, and set $\tau := t_{k} - t_{k-1}$.  
At time $t_k$ we compute $\varphi^k_h,\mu^k_h,\sigma^k_h \in V^k_h$ such that for all $v \in V^k_h$ it holds 
\begin{subequations}
\label{eq:FD:scheme}
\begin{alignat}{3}
\notag (\varphi^k_h,v) + \tau (m(I_h^k\varphi^{k-1}) \nabla \mu^k_h,\nabla v) & = (I_h^k \varphi^{k-1},v)  + \tau \PP (f(I_h^k\varphi^{k-1})g(\sigma_h^k),v) \\
& \quad +\tau \chi(m(I_h^k\varphi^{k-1})\nabla \sigma^k_h, \nabla v),
\label{eq:FD:1_CH}\\
 \eps (\nabla \varphi_h^k, \nabla v)
  + \eps^{-1}( \Psi'(\varphi_h^k),v)^h & = (\mu_h^k,v), \label{eq:FD:2_CH}\\
(\sigma^k_h,v) + \tau (\nabla \sigma^k_h,\nabla v) & =  (I_h^k\sigma^{k-1},v) - \tau \CC(h(I_h^k\varphi^{k-1})\sigma^k_h,v),
  \label{eq:FD:3_NU}
\end{alignat}
\end{subequations}
where $(\phi,\psi) = \int_{\Omega} \phi \, \psi \, \mathrm{d}x$ denotes the $L^{2}(\Omega)$-inner product.
In \eqref{eq:FD:2_CH} we use the lumped integration $(u,v)^h = \int_\Omega I_h^k(uv)\, \mathrm{d}x$ 
for the integral involving $\Psi'(\varphi_{h}^{k})$.  
For $k=1$ we set $\varphi^0 := \Pi_h \varphi_0$ and $\sigma^0 := \Pi_h \sigma_0$, where
$\Pi_h$ denotes the $L^2$-projection onto $V^1_h$, and in simulations we choose $\eps = 0.05$.

For the potential $\Psi$ we use a relaxed \change{double obstacle} potential \cite{1991_BloweyElliott,2011_HintHT}:
\begin{align*}
  \Psi(\varphi) = \frac{1}{2}(1-\varphi^2) + \frac{s}{2}\Lambda_\rho(\varphi)
\end{align*}
for some constant $s \gg 0$, 
where $\Lambda^\prime_\rho(\varphi) = \lambda_\rho(\varphi) := \max_\rho(0,\varphi-1)+\min_\rho(0,\varphi+1)$.
Note that $\min_\rho(\cdot,\cdot)$ and $\max_\rho(\cdot,\cdot)$ are regularizations of $\min(\cdot,\cdot)$ and $\max(\cdot,\cdot)$ according to \cite[(2.5)]{2011_HintK} such that the resulting potential $\Psi$ belongs to $C^{3,1}(\R)$.   For our numerical simulations, we fix $s = 10^4$ and $\rho = 0.001$.

The functions $f$, $g$, $h$, and $m$ are chosen as in \cite{KL}: let $q(s) := \min(1, \max(s, -1))$, then we define
\begin{align*}
 f(s) &= \frac{1}{2}(\cos (\pi q(s)) +1 ),\quad  h(s) = \frac{1}{2}\left(\sin\left(\frac{\pi}{2} q(s) \right) + 1\right), \quad
  m(s) = (m_1-m_0)f(s) + m_0,
\end{align*}
and for some $M > 0$, we consider $g$ such that $g(s) = 0$ if $ s\leq 0$, $g(s) = M$ if $s \geq M$, and
\begin{align*}  g(s) =   \begin{cases}
  s^2(-\theta^{-2}s+2\theta^{-1}) & \mbox{ if } 0 < s<\theta,\\
  s & \mbox{ if } \theta \leq s \leq M-\theta,\\
  -\theta^{-2}(s-M)^3 -2\theta^{-1}(s-M)^2 + M & \mbox{ if } M-\theta < s < M.
  \end{cases}
\end{align*}
The value $M$ can be viewed as the maximum amount of nutrition that can be used for proliferation.  For simulations we choose
\begin{align*}
M = 10, \quad \theta = 0.01, \quad m_1 = 0.05, \quad m_0 = 5 \cdot 10^{-6}.
\end{align*}
Let us motivate the reason for choosing a very small number for $m_0 = m(-1)$.  
If we consider the mobility $m(\varphi) \equiv 1$, then \eqref{varphi} reduces to
\begin{align}\label{phi:const:mob}
\varphi_t = \Lap \mu - \chi \Lap \sigma + \PP f(\varphi) g(\sigma),
\end{align}
and in preliminary tests not reported here we observe the sudden appearance of new tumour cells in the host cell region $\{\varphi(x,t) = -1\}$ that are far away from the main tumour region $\{\varphi(x,t) = 1\}$.  
We attribute this non-physical effect to the chemotaxis term $-\chi \Lap \sigma$ in \eqref{phi:const:mob}, since small variations of the nutrient $\sigma$ in the host cell region can induce growth of the tumour cells there.  
Therefore, we employ a non-constant mobility $m(\varphi)$ such that $m(-1)$ is nearly degenerate to limit the chemotaxis mechanisms in the host cell regions; this has also been used in previous works \cite{GLNS,GLSS,Wise}.

At this point we note a discrepancy between the proposed numerical set-up and the theoretical results.  That is, our numerical domain (denoted by $\Omega_h$ in this paragraph) does not fulfil the requirement outlined in \eqref{ass:C4}.  However, in the subsequent simulations (see \Cref{fig:num:artData:phi_d}(a) and (b)), the tumour region $\{\varphi(x,t) = 1 \}$ is located far away from the computational boundary $\pd \Omega_h$, with $(\varphi, \mu, \sigma)$ attaining nearly constant values in a neighbourhood of $\pd \Omega_h$.   This allows us to treat $\Omega_h$ as a Lipschitz subset of a larger domain $\Omega$ that fulfils \eqref{ass:C4}, so that the theoretical results are valid there.

%

\section{Sequential Monte Carlo with tempering}
\label{sec:SMC} 
To solve the Bayesian inverse problem we apply particle-based methods. 
In particular, we approximate the posterior measure $\mu^y$ by a discrete measure of the form
\begin{equation*}
\widehat{\mu}^y = \sum_{i = 1}^n w^{(i)} \delta_{u^{(i)}}.
\end{equation*}
Here, $w^{(i)} > 0$, $i = 1,\dots,n$, are positive weights that sum to one, and $\{u^{(i)}\}_{i=1}^n \in X^n$ is an ensemble of particles.
In the following, we briefly review two methods that are popular in Bayesian statistics and Bayesian inversion, namely \emph{importance sampling} and \emph{sequential Monte Carlo (SMC)}, see \cite{Agapiou2015,Beskos2015,DelMoral2006,Kantas2014} for more details.

 \subsection{Importance sampling}
Let $Q: X \rightarrow \mathbb{R}$ be a function that is square-integrable with
respect to the posterior measure.
\emph{Importance sampling} is based on the following identity
\begin{equation*}
\int Q \, \mathrm{d} \mu^y = \int Q  \frac{\mathrm{d} \mu^y}{\mathrm{d} \mu_0} \, \mathrm{d}\mu_0 = \frac{1}{ Z(y)} \int Q(u) \exp(-\Phi(u;y)) \, \mathrm{d}\mu_0(u)
\end{equation*}
which is a consequence of Bayes' formula \eqref{RN}.  
The above identity tells us that we can replace integrals given w.r.t.~the posterior by integrals w.r.t.~the prior.  While we are typically not able to sample (independently) from the posterior measure, it is often possible to sample from the prior.

We apply standard Monte Carlo techniques to approximate the normalisation constant $Z(y)$ and the integral $$\int Q(u) \exp(-\Phi(u;y))\, \mathrm{d}\mu_0(u)$$ using $n$ samples of the prior $\mu_0$.
This is equivalent to integrating $Q$ w.r.t.~a specific discrete measure given by
\begin{equation*}
\widehat{\mu}^y = \sum_{i = 1}^n w^{(i)} \delta_{u^{(i)}}, \ \ w^{(i)} = \frac{\exp(-\Phi(u^{(i)};y))}{\sum_{j = 1}^n\exp(-\Phi(u^{(j)};y))}, \ \ i=1,\dots,n, \ \  u^{(1)},\dots,u^{(n)} \sim \mu_0  \text{ i.i.d.}.
\end{equation*}
The random variables $u^{(1)},\dots,u^{(n)}$ are measurable functions mapping from a probability space $(\Omega', \mathcal{F}',\mathbb{P})$ to $(X, \mathcal{B}X)$, where we recall that $\mathcal{B}X$ denotes the Borel-$\sigma$-algebra of $X$.
Hence the measure $\widehat{\mu}^y$ is a measure-valued random variable.
It is possible to show that $\widehat{\mu}^y$ converges weakly to the posterior measure ${\mu}^y$ as $n \rightarrow \infty$. 
In fact, \cite[Thm. 2.1]{Agapiou2015} states that
\begin{equation} \label{eq:CV_Bound}
\sup_{\|Q\|_\infty \leq 1}\left(\int \left(\int Q \mathrm{d}\mu^y - \int Q \mathrm{d}\widehat{\mu}^y\right)^2 \mathrm{d}\mathbb{P}\right)^{1/2} 
\leq 2  \left(\frac{ 1+ {{\mathrm{cv}}^2}}{n}\right)^{1/2},
\end{equation}
where the quantity
\begin{align*}\label{cv}
{\mathrm{cv}} := \left ( \frac{\int \exp(-\Phi(\cdot  \, ;y))^2 \,
\mathrm{d}\mu_0}{\left(\int \exp(-\Phi(\cdot \, ;y))\, \mathrm{d}\mu_0\right)^2}-1 \right)^{\frac{1}{2}}
\end{align*}
is the \textit{coefficient of variation} of the update density $\exp(-\Phi)$.

\subsection{Sequential Monte Carlo}
Importance sampling can be inefficient in Bayesian inversion, in particular, when the parameter space is high-dimensional or the data (resp.~the likelihood) is highly informative.   In these cases the posterior can be concentrated in a small region of the parameter space.  
In contrast, the prior is typically not concentrated, and so in this setting a large number of prior samples is required to obtain a useful approximation of the posterior measure.  
However, for every prior sample we need one evaluation of the potential $\Phi$, and in practice this can lead to a massive number of (expensive) model evaluations.

\emph{Sequential Monte Carlo} overcomes this issue by constructing a sequence of measures $\{\mu_k\}_{k = 0}^{K}$ starting with the prior, and slowly approaching the posterior $\mu_K \approx \widehat{\mu}^y$.
The sequence is constructed such that each of the measures $\mu_k$ allows an efficient importance sampling approximation of the measure $\mu_{k+1}$ for $k=0,\ldots,K-1$.

Observe that the noise covariance has an impact on the concentration of the posterior.  Hence, we construct our sequence by starting with a highly up-scaled noise covariance and proceed by scaling the noise-level to the actual level, i.e., {for $k = 1, \dots, K$},
\begin{equation*}
\frac{\mathrm{d}\mu_k}{\mathrm{d}\mu_0} \propto \exp(-\beta_k\Phi(\cdot \, ; y)) =\gamma_k = \exp \left( -\frac{1}{2}\normY{(\beta_k^{-1}\Gamma)^{-1/2}\G(u)}^2-\langle \G(u),(\beta_k^{-1}\Gamma)^{-1} y\rangle_{Y}\right), 
\end{equation*}
where $\{\beta_k \}_{k=0}^K$ is an increasing sequence starting at $\beta_0 = 0$ and finishing at $\beta_K = 1$.  Originating in statistical thermodynamics, this procedure is often referred to as \emph{tempering}.

Sequential Monte Carlo proceeds in the following way. 
First, $n$ samples $\{u^{(1)},\ldots,u^{(n)}\}$  are drawn independently from the prior $\mu_0$.  
Then, the samples are weighted with $\gamma_1$ to approximate $\mu_1$ using the importance sampling idea.  
This gives the discrete measure  $\widehat{\mu}_1$.  
We eliminate particles with smaller weights from the ensemble by \emph{resampling} the ensemble, drawing new, equally weighted samples $\{u^{(1)},\ldots,u^{(n)}\}$ from $\widehat{\mu}_1$.  
To distribute the particles more evenly in the parameter space, we pass the samples through a \emph{Markov kernel} that is stationary w.r.t.~$\mu_1$. 
This gives a new set of samples $\{u^{(1)},\ldots,u^{(n)}\}$ that is approximately $\mu_1$-distributed.   
The Markov kernel is typically given by a Markov chain Monte Carlo (MCMC) sampler, see \cite[\S 7-10]{Robert2004} for an introduction to MCMC and  \cite{Beskos2008,Cotter2013} for a discussion of MCMC methods for Bayesian inversion.

We then proceed iteratively for $k = 2,\ldots,K$.  
Each sample $\{u^{(1)},\ldots,u^{(n)}\}$ is weighted according to the update density $\gamma_k/\gamma_{k-1}$.  
Then, we resample the particles and apply a Markov kernel that is stationary w.r.t. to $\mu_k$.  When $k = K$, we stop the process and obtain $\widehat{\mu}_K =: \widehat{\mu}^y$.


\subsection{Adaptivity for the tempering}
It is not intuitively clear how to choose the tempering sequence $\{\beta_k \}_{k=0}^K$.  
A typical approach is induced by the importance sampling error and its connection to the coefficient of variation $\mathrm{cv}$ of the update density in \eqref{eq:CV_Bound}.
The bound in \eqref{eq:CV_Bound} tells us that if ${\mathrm{cv}}$ is small, then the accuracy of the importance sampling approximation is high, even if the number of samples $n$ is small.
On the other hand, a small ${\mathrm{cv}}$ leads to more sequential Monte Carlo steps, as $\{\beta_k \}_{k=0}^K$ needs more steps to reach $\beta_K = 1$.
Each step also decreases the accuracy, since more particle approximations are performed.  
This issue is termed \emph{path degeneracy} and has been observed and discussed for instance in \cite{Andrieu1999} and \cite[\S 5.1.2]{Latz2018}.  
To date finding an optimal ${\mathrm{cv}}$ is still an open research question, where by \emph{optimal} we mean that both the ${\mathrm{cv}}$ and the number of update steps $K$ are minimised.

In practice, we use a simple parameter fitting approach to choose $\{\beta_k \}_{k=0}^K$ adaptively such that the (sample) coefficient of variation in each update step equals some target $\overline{\mathrm{cv}} > 0$ that has been chosen a priori.  
One can show that this is in fact a root finding problem in one spatial dimension.  
We mention that the adaptive algorithm introduces a bias into the estimation of the model evidence.
In \cite{Beskos2016} it is shown that the adaptive SMC method is convergent. 
 
\section{Numerical examples}
\label{sec:num}
In this section we apply the SMC approach described in \Cref{sec:SMC} 
for the identification of the  parameters in the numerical approximation \eqref{eq:FD:scheme}  of the tumour model \eqref{CH}. 
In the remainder of this section we always work with fully discrete functions, 
and neglect the subscript $h$ for the finite element approximations introduced in \eqref{eq:FD:scheme}.

The implementation is done using \texttt{C++}. 
We use the finite element library FEniCS 1.6.0 \cite{fenics_book} together with
the PETSc 3.6.4 \cite{petsc_webpage} linear algebra backend  and the direct solver MUMPS 5.0.0 \cite{mumps}. The meshes are generated
and adapted using ALBERTA 3.0.1 \cite{alberta_book}.  
As we mention in \Cref{sec:intro}, we only present results for the more complex setting $(a)$.

We use the set-up from \cite{KL}: the domain $\Omega = (-5, 5) \times (-5, 5)$ with initial conditions
\begin{align*}
  \varphi_0(x) = \Phi_0( \eps^{-1}(1-\|x\|_{l^8} )), \quad \eps = 0.05, \quad   \sigma_0(x) \equiv 1 ,
\end{align*}
where for $z_0 = \arctan(\sqrt{s-1})$, $s = 10^4$,
\begin{align*}
  \Phi_0(z) := 
  \begin{cases}
    \sqrt{\frac{s}{s-1}}\sin(z) & \mbox{ if } 0 \leq z \leq z_0,\\
    \frac{1}{s-1}\left(s-\exp( \sqrt{s-1}(z_0-z)  )\right) & \mbox { if } z \geq z_0,
  \end{cases} \quad \Phi_0(z) = - \Phi_0(-z) \mbox{ if } z < 0.
\end{align*}
We simulate system \eqref{eq:FD:scheme} with the parameter values
\begin{align}\label{true:para}
\PP = 7, \quad \chi = 120, \quad \CC = 2
\end{align}
until the final time $T=4$ with time steps of size $\tau = 0.05$ to obtain $\varphi_d(x) = \varphi(x,T)$.
Then we add normally distributed noise with mean $0$ and standard deviation $0.1$ to every degree of freedom of $\varphi_d$.  
The resulting function is considered as the (synthetic) data $y$. 
In \Cref{fig:num:artData:phi_d} we display $\varphi_0$ (left), $\varphi_d$ (middle) and $y$ (right). 
\begin{figure}
  \centering
  \subfloat[$\varphi_0(x)$]{
  \includegraphics[width=0.25\textwidth]{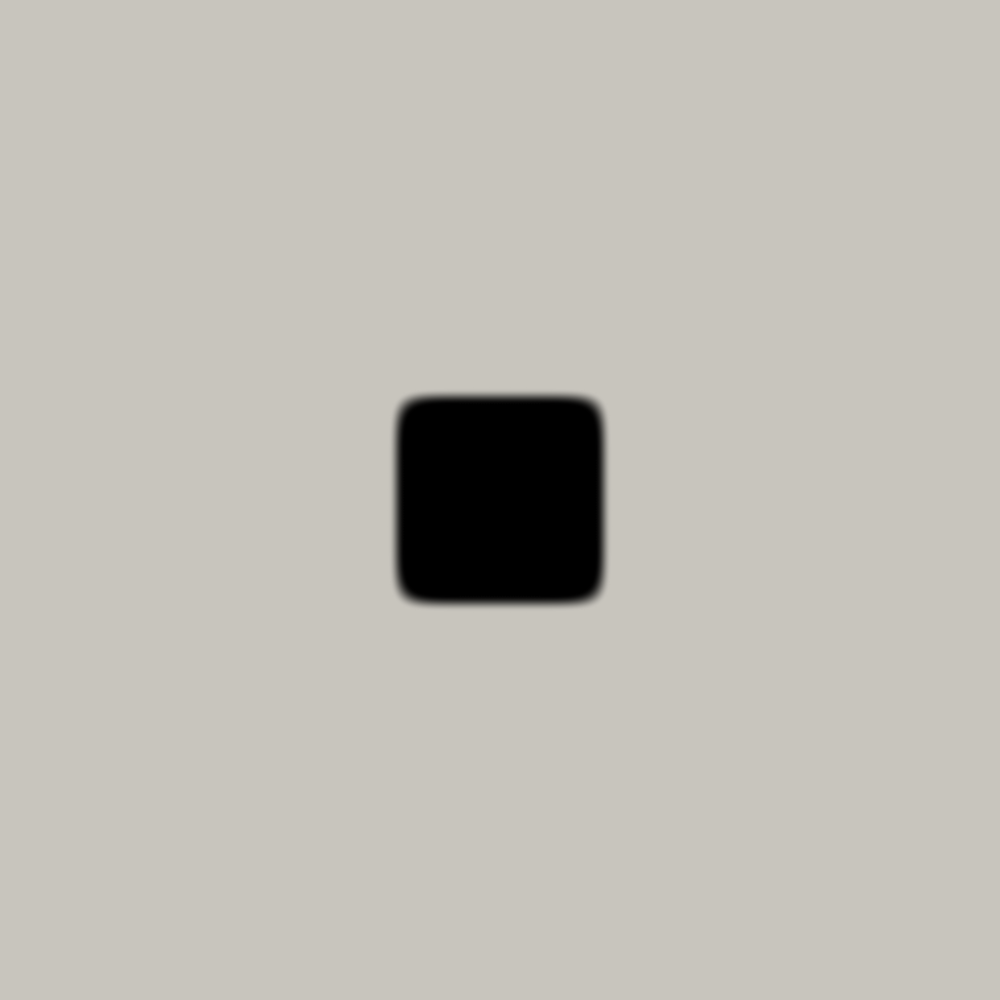}  
  }
  \hspace{0.5cm}
  \subfloat[$\varphi_d(x)$]{
  \includegraphics[width=0.25\textwidth]{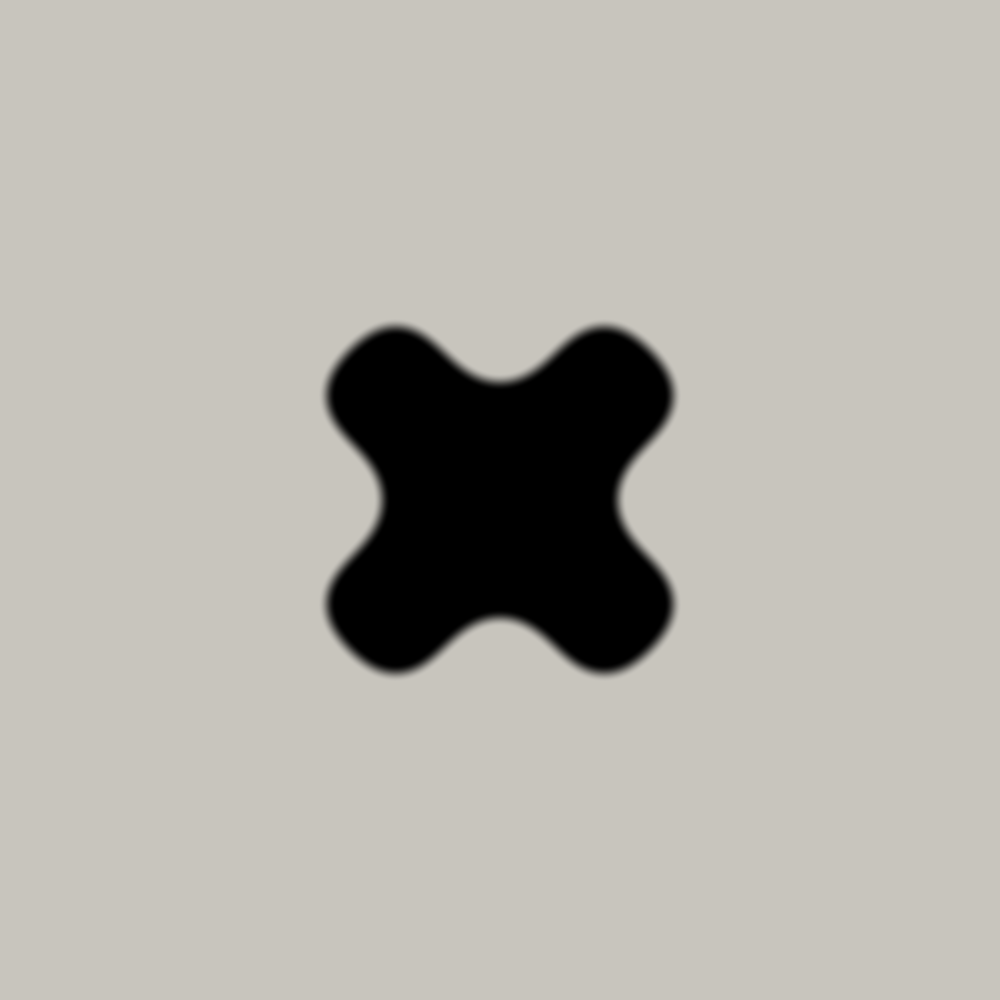}
  }
  \hspace{0.5cm}
  \subfloat[$y(x)$]{ 
  \includegraphics[width=0.25\textwidth]{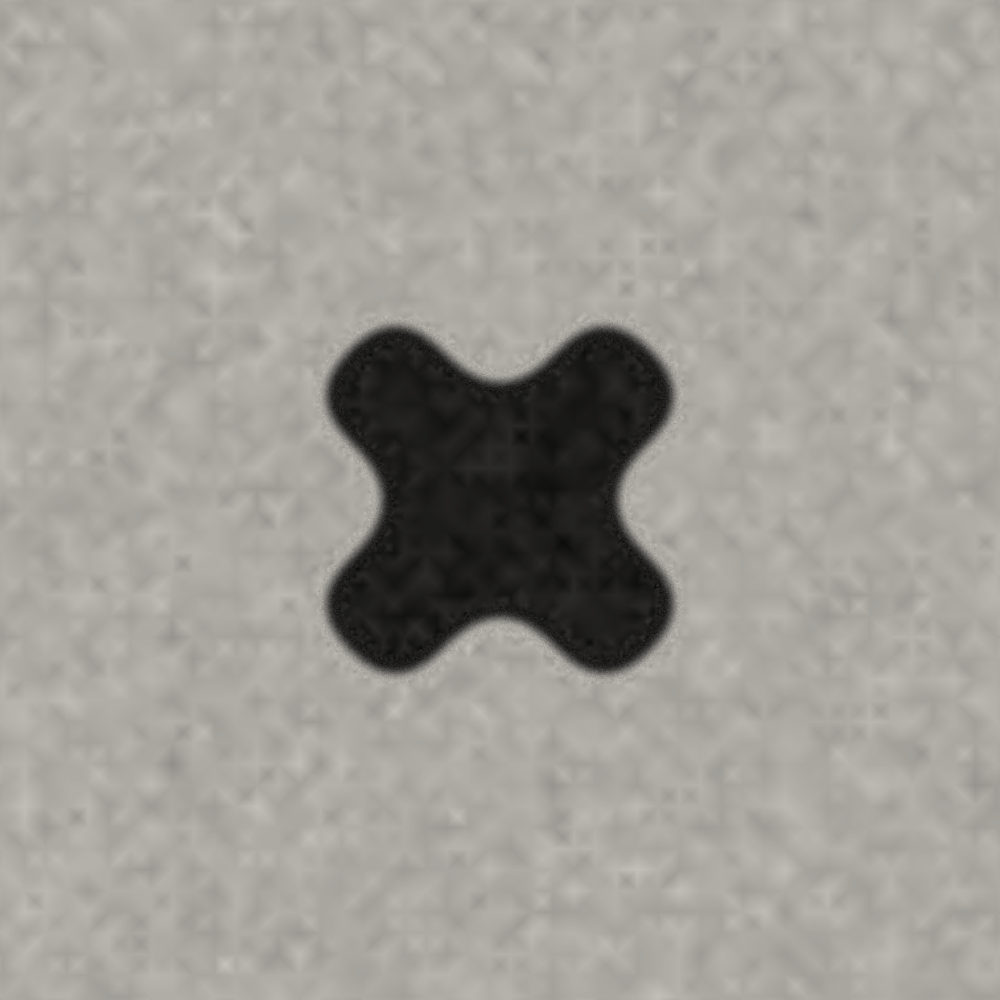}
  }
  \caption{The initial condition $\varphi_0(x)$ where the tumour occupies the black region and the host cells occupy the gray region (left), the uncorrupted data $\varphi_d(x) = \varphi(x,4)$ (middle) and the data $y(x)$ corrupted by noise (right).}
  \label{fig:num:artData:phi_d}
\end{figure}

We consider the likelihood defined in \eqref{likelihood} with given pointwise variance $\sigmaa^2 = 0.1$.  The prior is the product of truncated normal distributions  (see \Cref{ssec:priors}) with parameters
\begin{align*}
m_{\PP} = 5, \quad \sigma_{\PP} = 2, \quad m_{\chi} = 100, \quad \sigma_{\chi} = 40, \quad m_{\CC} = 5, \quad \sigma_{\CC} = 2.
\end{align*} 
Guided by medical rationality, we choose the upper bounds for the parameter space $X$ defined in \eqref{defn:X} to be
\begin{align*}
\PP_{\infty} = 10, \quad \chi_{\infty} = 200, \quad \CC_{\infty} = 10.
\end{align*}
\change{Note that the variances $\sigma_{\PP}^2, \sigma_{\chi}^2, \sigma_{\CC}^2$ refer to the underlying non-truncated Gaussian measures.
The variances of the (truncated Gaussian) prior measure are recorded in \Cref{tab:num:artData:prior_covariance}. 
The prior mean values coincide with the means of the non-truncated Gaussian measures.
}

\begin{table}
\centering

\input{images/cmcd2_prior_covariance.tex}  
\caption{\change{The prior covariances of the parameters $\PP$, $\chi$, and $\CC$.}}
\label{tab:num:artData:prior_covariance}
\end{table}

We use $n = 400$ particles to approximate the posterior measure and choose $\overline{\mathrm{cv}} = 0.25$ for the tempering steps.  
\change{Looking at the upper bound in \eqref{eq:CV_Bound}, this gives an effective sample size $n_{\text{eff}}\approx \lfloor n/(1+\overline{\mathrm{cv}}^2) \rfloor=\lfloor 400/(17/16) \rfloor = 376$, and a relative error of about 10\% in the measure approximation.}
Note that a single simulation run for one set of parameters in model \eqref{CH} takes around 5 minutes.  Hence, to be able to conduct numerical experiments within a reasonable amount of time, we do not use a large number of particles within SMC.

In \Cref{fig:num:artData:posterior_histogram_P_chi_C} we depict the marginal posterior distributions of $\PP$, $\chi$ and $\CC$, respectively.  The corresponding posterior sample means are
\begin{align}\label{post:mean}
m_{\PP}^y = 6.89, \quad m_{\chi}^y = 120.19, \quad m_{\CC}^y = 1.99
\end{align}  
which agree well with the parameter values \eqref{true:para} that have been used to generate the data $y$.

\begin{figure}
  \centering 
  \input{images/cmcd2_400p_marginal_P.tikz.tex}%
  \hspace{0.5cm}
   \input{images/cmcd2_400p_marginal_chi.tikz.tex}%
   \hspace{0.5cm}
  \input{images/cmcd2_400p_marginal_C.tikz.tex}%
  \caption{The posterior marginal distributions for $\PP$, $\chi$, and $\CC$ with corresponding posterior mean $m_{\PP}^y = 6.89$, $m_{\chi}^y = 120.19$, and $m_{\CC}^y = 1.99$, respectively.  The black lines indicate parts of the prior distributions.
  }
  \label{fig:num:artData:posterior_histogram_P_chi_C}
\end{figure}
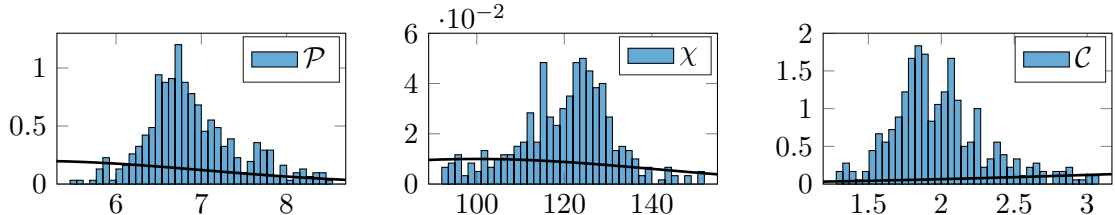   

In addition, we compute the \emph{maximum a posteriori (MAP)} estimator, that is, the global maximum of the density of the posterior measure.
\change{This connects the Bayesian approach for parameter identification with the classical, regularised minimisation approach, see e.g. \cite{SC,KS,Tarantola}.
Indeed, }
it is well known that for Gaussian priors the MAP is the solution of a Tikhonov regularised least-squares inverse problem \cite[\S 2.2]{Stuart}.
Hence the MAP can be computed by the method proposed in \cite{KL} by choosing suitable values for the individual Tikhonov weights associated with the parameters.  
For our example we obtain the MAP estimates
\begin{align}\label{MAP}
\MAP_\PP = 7.0188, \quad \MAP_\chi = 106.9212, \quad \MAP_\CC = 2.4836,
\end{align}
using the numerical code from \cite{KL} with the prior mean as initial value in the optimisation. 
We observe that the MAP estimate for both the chemotaxis parameter $\chi$ and the consumption rate $\CC$ does not agree very well with the underlying true parameter value \eqref{true:para}.  
This is in contrast to the Bayesian posterior mean estimate \eqref{post:mean}; we think that this might be due to the presence of local maxima.  
Furthermore, we obtain the same MAP estimate \eqref{MAP} when we start the numerical code with the underlying true parameter values \eqref{true:para}.
\change{
However, for small noise levels, the MAP actually corresponds to the underlying true parameter values \eqref{true:para}, 
see the results in \cite[\S 7.4]{KL}.
The larger noise level in our study has the effect of a stronger Tikhonov regularisation and thus a stronger influence of the prior measure.
}

In \Cref{tab:num:artData:posterior_covariance} we report the posterior covariances of the parameters.
\change{We observe that the posterior variances in \Cref{tab:num:artData:posterior_covariance} are about 3--8\% of the size of the prior variances in \Cref{tab:num:artData:prior_covariance}. 
The variance reduction corresponds to a reduction of uncertainty in the model parameters.
Equivalently, it can be interpreted as an information gain during the learning process from prior to posterior. 
See also \Cref{fig:num:artData:posterior_histogram_P_chi_C}, where the concentration of the posterior (histogram) with respect to the prior (graph) is depicted.
This justifies the Bayesian approach to the Cahn--Hilliard parameter identification problem.
 }
 
\change{Next we provide a model-based discussion of the covariances.}
We observe that $\PP$ and $\chi$ are negatively correlated, as are $\chi$ and $\CC$, but $\PP$ and $\CC$ have a positive correlation.  The negative correlation between $\PP$ and $\chi$ can be attributed to the fact that
both parameters cause tumour growth, albeit through different mechanisms; namely, $\PP$ leads to undirected growth, while $\chi$ gives directional growth depending on nutrient concentration.  
It is likely that neither parameter can be large at the same time in order to obtain a tumour of a comparable size to the data.  
Similarly, larger values of $\CC$ lead to a larger nutrient gradient, and so large values of $\chi$ would likely amplify the directed growth of the tumour.
The negative correlation between $\CC$ and $\chi$ is a means to match the simulations more closely with the data.
On the other hand, larger values of $\CC$ imply that nutrients are consumed at a faster rate, and thus on the growing front $\{\abs{\varphi(x,t)} < 1\}$ the level of nutrients is lower compared to regions away from the growing front.
Hence to maintain growth in regions of lower nutrient concentration a larger value of $\PP$ is desirable, which may attribute to a positive correlation between $\PP$ and $\CC$.

\begin{table}
\centering

\input{images/cmcd2_400p_covariance.tex}  
\caption{The posterior covariances of the parameters $\PP$, $\chi$, and $\CC$.}
\label{tab:num:artData:posterior_covariance}
\end{table}

Finally, let $\varphi^K$ denote the posterior output of the SMC algorithm, and denote by $\varphi^K_m$ and $\varphi^K_\sigma$ the pointwise (posterior) mean and variance of $\varphi^K$, respectively.  
In the left panel of \Cref{fig:num:artData:phi_K__mean_variance} we display the zero level lines of $\varphi^K_m$ (in black) and of $\varphi_d$ (in white) superimposed on a plot of $\varphi^K_m$.  
We see that the zero level lines match quite well, so that our output is close to the original synthetic data $\varphi_d$.  
In the right panel of \Cref{fig:num:artData:phi_K__mean_variance} we plot the pointwise variance $\varphi^K_\sigma$ together with the $\pm 1$ isolines of $\varphi^K_m$.  
The maximum of $\varphi^K_\sigma$, attained in the black regions, is of order $0.4$.

 \begin{figure}
   \centering
   \subfloat[Pointwise mean $\varphi^K_m$]{
   \includegraphics[width=0.4\textwidth]{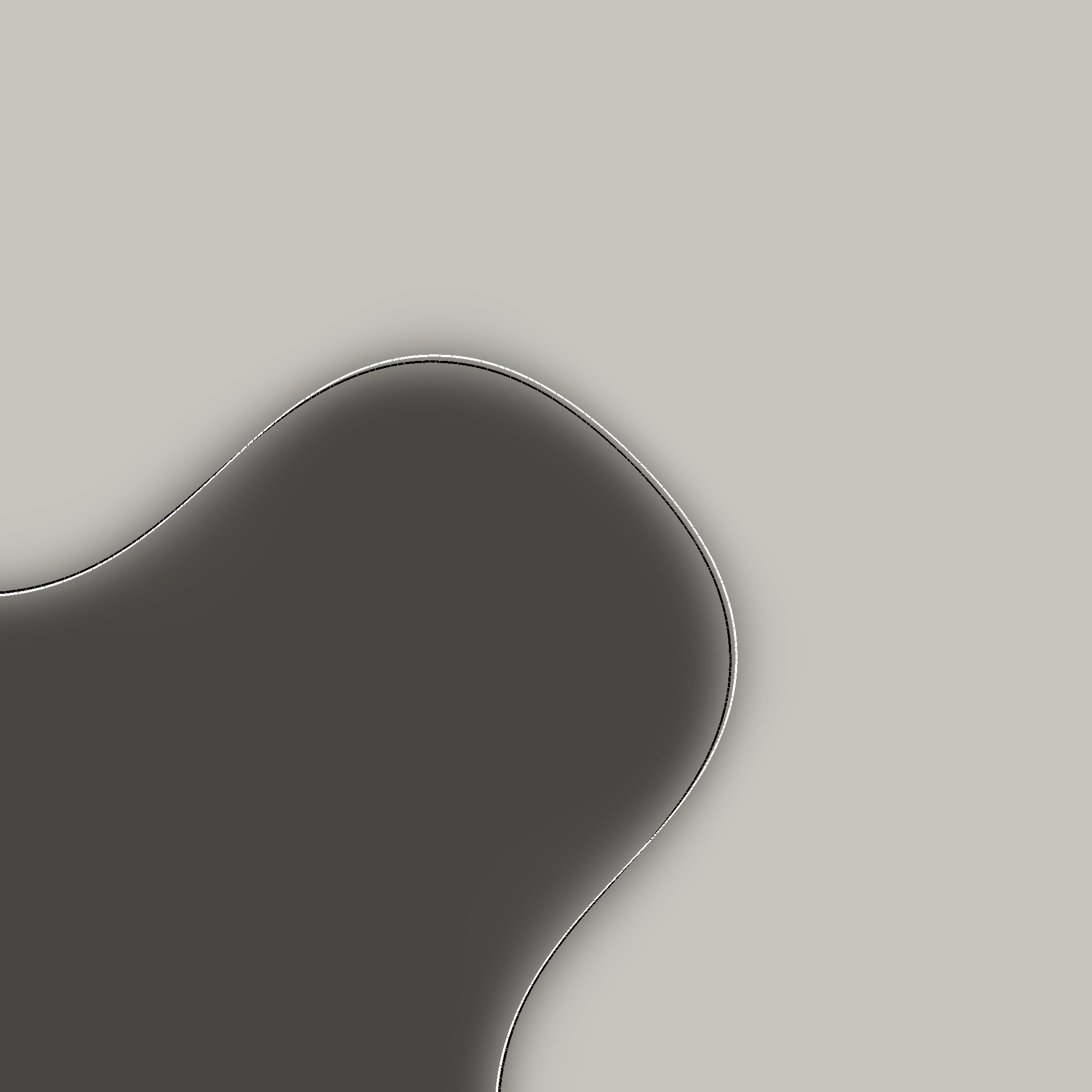}
   }
   \hspace{0.95cm}
   \subfloat[Pointwise variance $\varphi^K_\sigma$]{
   \includegraphics[width=0.4\textwidth]{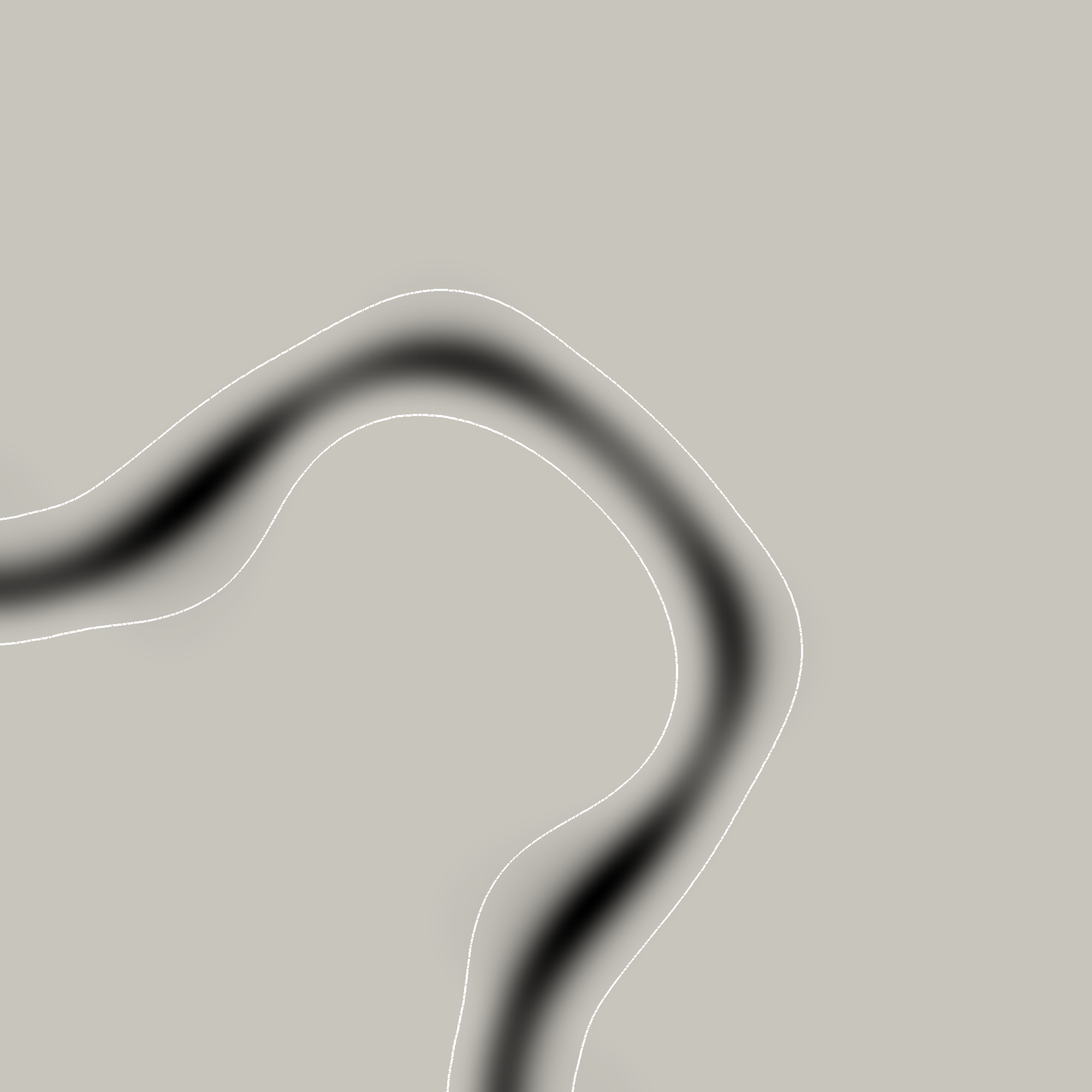}
  } 
   \caption{(Left) The pointwise mean $\varphi^K_m$ together with the zero level lines for $\varphi^K_m$ (in black) and for synthetic data $\varphi_d$ (in white).  The value of $\varphi^K_m$ is close to 1 in the darker regions and close to $-1$ in the lighter gray regions.  (Right) The pointwise variance $\varphi^K_\sigma$ together with the $\pm1$ isolines of $\varphi^K_m$ (in white).  The magnitude of $\varphi^K_\sigma$ is of order $0.4$ in darker regions, while in the light gray regions the magnitude is close to zero.  Due to symmetry of the configuration, we only show a quarter of the computational domain.}
    \label{fig:num:artData:phi_K__mean_variance}
 \end{figure}


\section{Discussion}\label{sec:Discussion}
In this paper we studied a Bayesian inverse problem to identify model parameters in a diffuse interface model for tumour growth.
We improved strong well-posedness results for the model \eqref{CH} and proved the well-posedness of the posterior measure for observational settings involving both finite and infinite-dimensional data spaces. 
For the numerical implementation we employed sequential Monte Carlo with tempering in combination with a finite element discretisation to approximate the posterior measure of the unknown parameters.  
We conducted a numerical experiment against a synthetic data set observing the tumour configuration at a fixed time.  
To finish we discuss possible directions of further work.

\subsection{Other model variants}
In some situations, the nutrient diffusion timescale ($\sim$ minutes) occurs at a much faster rate than the tumour doubling timescale ($\sim$ days).
Hence it is appropriate to neglect the time derivative in \eqref{sigma}, leading to a quasi-static evolution
\begin{align*}
0 = - \Lap \sigma + \CC h(\varphi) \sigma.
\end{align*}
However, the loss of the time derivative $\sigma_t$ implies that the regularity for the $\sigma$ variable is reduced.
Therefore some non-trivial modifications are needed to obtain the continuous dependence of $\varphi$ in the $C^0([0,T];C^0(\overline{\Omega}))$-norm (for setting $(a)$) and in the $C^0([0,T];L^2(\Omega))$-norm (for setting $(b)$) in order for the  resulting Bayesian inverse problem to be well-posed.  We leave this verification for future research and remark that ideas in \cite{GL} may be helpful.

Furthermore, many of the earlier diffuse interface tumour models include a notion of cellular velocity, which affixes the system \eqref{CH} with a Darcy-type equation and introduces convection terms for $\varphi$ and $\sigma$.  It is reported in \cite{GLNS} that such models produce biologically more realistic results compared to models without fluid velocity for the situation involving multiple species of cells.  We do not consider this extension in our present setting with two components (tumour and host cells), 
since the differences are less significant compared to the multispecies case. 
Further research is needed to improve the current analytical results for the models with Darcy flow so that an analogue of \eqref{ctsdep} is available.  Then, a similar analysis for the Bayesian inverse problem can be performed.

\subsection{Surrogates}
Bayesian inversion for the Cahn--Hilliard model \eqref{CH} is very expensive since the repeated evaluation of the likelihood requires forward solves of \eqref{CH} for many different parameter configurations and initial states. 
The computational burden can be reduced by constructing surrogates for $\varphi$ which can be evaluated cheaply without the need to run an expensive forward solve.  For Bayesian inversion a number of surrogates have been studied, e.g.\ Gaussian process models \cite{KH:2001}, or generalised polynomial chaos surrogates \cite{MN:2009,MNR:2007}.  
We point out that the solution of the Cahn--Hilliard model \eqref{CH} depends continuously on the parameters ($\mathcal{P}$, $\chi$, $\mathcal{C}$), and so we envision that it is feasible to construct smooth, polynomial based surrogates, or sparse grid surrogates.
While this has been \change{investigated for} the classical elliptic Bayesian inverse problem \cite{SchillingsSchwab:2013, SchwabStuart:2012}, this is not the case for Cahn--Hilliard models such as \eqref{CH}.
In addition, the error and convergence analysis for surrogates in Bayesian inversion is far from complete (see e.g.\ \cite{ST:2018,YanZhang:2017} for recent studies), and requires further work.

{
\bibliographystyle{plain}
\bibliography{KLLU_revised}}


\end{document}

%% file: images/gaussian.pgf
\pgfmathdeclarefunction{gauss}{2}
{%
  \pgfmathparse{1/(#2*sqrt(2*pi))*exp(-((x-#1)^2)/(2*#2^2))}%
}

%% file: images/cmcd2_prior_covariance.tex

\begin{tabular}{c|rrr}
      &$\PP$    & $\chi$ & $\CC$\\
      \hline
$\PP$  &$3.6135$&$0$ &$0$\\
$\chi$ &        &$1445.4095$&$0$\\
$\CC$  &        &          &$3.6135$
\end{tabular}

%% file: images/cmcd2_400p_marginal_P.tikz.tex
%
%
\definecolor{mycolor1}{rgb}{0.00000,0.44700,0.74100}%
\begin{tikzpicture}

\begin{axis}[%
width=3.8cm,
height=2cm,
at={(0.0cm,0.0cm)},
scale only axis,
xmin=5.306,
xmax=8.694,
ymin=0,
ymax=1.3,
domain=5.3:8.7,
axis background/.style={fill=white},
legend style={legend cell align=left,align=left,draw=white!15!black}
]
\addplot[fill=mycolor1,fill opacity=0.6,draw=black,ybar interval,area legend] plot table[row sep=crcr] {%
x	y\\
5.46	0.0324675324675325\\
5.537	0.0324675324675325\\
5.614	0\\
5.691	0.0324675324675325\\
5.768	0.12987012987013\\
5.845	0.227272727272727\\
5.922	0.0324675324675325\\
5.999	0.12987012987013\\
6.076	0.162337662337662\\
6.153	0.25974025974026\\
6.23	0.324675324675325\\
6.307	0.422077922077922\\
6.384	0.487012987012987\\
6.461	0.941558441558431\\
6.538	0.876623376623377\\
6.615	0.90909090909091\\
6.692	1.2012987012987\\
6.769	0.876623376623377\\
6.846	0.77922077922078\\
6.923	0.681818181818182\\
7	0.454545454545455\\
7.077	0.551948051948052\\
7.154	0.487012987012987\\
7.231	0.324675324675325\\
7.308	0.38961038961039\\
7.385	0.227272727272727\\
7.462	0.0974025974025963\\
7.539	0.194805194805195\\
7.616	0.357142857142857\\
7.693	0.292207792207792\\
7.77	0.292207792207792\\
7.847	0.064935064935065\\
7.924	0.162337662337662\\
8.001	0.0324675324675325\\
8.078	0.0974025974025974\\
8.155	0.12987012987013\\
8.232	0.064935064935065\\
8.309	0.0974025974025974\\
8.386	0.0974025974025974\\
8.463	0.0324675324675325\\
8.54	0.0324675324675325\\
};
\addlegendentry{$\PP$};

\addplot[line width=1pt] {gauss(5.0,2.0) };

\end{axis}
\end{tikzpicture}%

%% file: images/cmcd2_400p_marginal_chi.tikz.tex
%
%
\definecolor{mycolor1}{rgb}{0.00000,0.44700,0.74100}%
\begin{tikzpicture}

\begin{axis}[%
width=3.8cm,
height=2cm,
at={(0.0cm,0.0cm)},
scale only axis,
xmin=89,
xmax=155,
ymin=0,
ymax=0.06,
domain=89:155,
axis background/.style={fill=white},
legend style={legend cell align=left,align=left,draw=white!15!black}
]
\addplot[fill=mycolor1,fill opacity=0.6,draw=black,ybar interval,area legend] plot table[row sep=crcr] {%
x	y\\
92	0.00666666666666667\\
93.5	0.00833333333333333\\
95	0.0116666666666667\\
96.5	0.00333333333333333\\
98	0.00833333333333333\\
99.5	0.005\\
101	0.0133333333333333\\
102.5	0.00666666666666667\\
104	0.01\\
105.5	0.0116666666666667\\
107	0.0116666666666667\\
108.5	0.015\\
110	0.0166666666666667\\
111.5	0.0283333333333333\\
113	0.0166666666666667\\
114.5	0.0483333333333333\\
116	0.0266666666666667\\
117.5	0.0233333333333333\\
119	0.03\\
120.5	0.035\\
122	0.0483333333333333\\
123.5	0.05\\
125	0.045\\
126.5	0.0383333333333333\\
128	0.04\\
129.5	0.0266666666666667\\
131	0.0133333333333333\\
132.5	0.015\\
134	0.0133333333333333\\
135.5	0.01\\
137	0.00333333333333333\\
138.5	0.00666666666666667\\
140	0\\
141.5	0.00166666666666667\\
143	0.00666666666666667\\
144.5	0.00333333333333333\\
146	0\\
147.5	0.00333333333333333\\
149	0\\
150.5	0.005\\
152	0.005\\
};
\addlegendentry{$\chi$};

\addplot[line width=1pt] {gauss(100,40) };

\end{axis}
\end{tikzpicture}%

%% file: images/cmcd2_400p_marginal_C.tikz.tex
%
%
\definecolor{mycolor1}{rgb}{0.00000,0.44700,0.74100}%
\begin{tikzpicture}

\begin{axis}[%
width=3.8cm,
height=2cm,
at={(0.0cm,0.0cm)},
scale only axis,
xmin=1.19,
xmax=3.17,
ymin=0,
ymax=2,
domain=1.19:3.17,
axis background/.style={fill=white},
legend style={legend cell align=left,align=left,draw=white!15!black}
]
\addplot[fill=mycolor1,fill opacity=0.6,draw=black,ybar interval,area legend] plot table[row sep=crcr] {%
x	y\\
1.28	0.166666666666667\\
1.325	0.277777777777777\\
1.37	0.166666666666667\\
1.415	0.0555555555555556\\
1.46	0.166666666666667\\
1.505	0.444444444444443\\
1.55	0.666666666666668\\
1.595	0.555555555555554\\
1.64	0.722222222222223\\
1.685	0.88888888888889\\
1.73	1.22222222222222\\
1.775	1.66666666666666\\
1.82	1.83333333333334\\
1.865	1.72222222222222\\
1.91	0.833333333333335\\
1.955	1.05555555555556\\
2	1.22222222222222\\
2.045	1.66666666666667\\
2.09	1.11111111111111\\
2.135	0.500000000000001\\
2.18	0.555555555555551\\
2.225	1\\
2.27	0.222222222222223\\
2.315	0.333333333333331\\
2.36	0.555555555555556\\
2.405	0.38888888888889\\
2.45	0.222222222222223\\
2.495	0.333333333333334\\
2.54	0.111111111111111\\
2.585	0.222222222222223\\
2.63	0.277777777777778\\
2.675	0.166666666666667\\
2.72	0\\
2.765	0.166666666666665\\
2.81	0.166666666666667\\
2.855	0.222222222222223\\
2.9	0.0555555555555551\\
2.945	0.0555555555555556\\
2.99	0.111111111111111\\
3.035	0.111111111111111\\
3.08	0.111111111111111\\
};
\addlegendentry{$\CC$};

\addplot[line width=1pt] {gauss(5.0,2.0) };

\end{axis}
\end{tikzpicture}%

%% file: images/cmcd2_400p_covariance.tex

\begin{tabular}{c|rrr}
      &$\PP$    & $\chi$ & $\CC$\\
      \hline
$\PP$  &$0.3038$&$-4.1340$ &$0.1369$\\
$\chi$ &        &$124.1128$&$-3.1234$\\
$\CC$  &        &          &$0.1134$
\end{tabular}